\documentclass[12pt,a4paper]{article}

\usepackage{amsmath,amsthm,amssymb,color}
\usepackage[authoryear]{natbib}
\usepackage{booktabs}

\usepackage{bbm}
\usepackage{mathrsfs}
\usepackage[breaklinks=true]{hyperref}
\usepackage{graphicx}
\usepackage{tikz}
\usetikzlibrary{arrows, automata}
\usetikzlibrary{calc}
\usetikzlibrary{positioning}

\numberwithin{equation}{section}
\allowdisplaybreaks[4]

\theoremstyle{plain}
\newtheorem{theorem}{Theorem}[section]
\newtheorem{proposition}{Proposition}[section]
\newtheorem{lemma}{Lemma}[section]
\newtheorem{corollary}{Corollary}[section]
\newtheorem{claim}{Claim}[section]

\theoremstyle{definition}

\newtheorem{remark}{Remark}[section]

\makeatletter

\newcount\minute
\newcount\hour
\newcount\hourMins
\def\now{%
\minute=\time%
\hour=\time \divide \hour by 60%
\hourMins=\hour \multiply\hourMins by 60%
\advance\minute by -\hourMins%
\zeroPadTwo{\the\hour}:\zeroPadTwo{\the\minute}%
}
\def\zeroPadTwo#1{\ifnum #1<10 0\fi#1}

\renewcommand{\cite}{\citet}

\def\^#1{\ifmmode {\mathaccent"705E #1} \else {\accent94 #1} \fi}
\def\~#1{\ifmmode {\mathaccent"707E #1} \else {\accent"7E #1} \fi}

\def\*#1{#1^\ast}
\edef\-#1{\noexpand\ifmmode {\noexpand\bar{#1}} \noexpand\else \-#1\noexpand\fi}
\def\>#1{\vec{#1}}
\def\.#1{\dot{#1}}

\def\atop{\@@atop}
\def\*#1{\mathscr{#1}}

\renewcommand{\leq}{\leqslant}
\renewcommand{\geq}{\geqslant}

\newcommand{\eq}{\eqref}

\newcommand{\IE}{\mathbbm{E}}

\def\be#1{\begin{equation*}#1\end{equation*}}
\def\ben#1{\begin{equation}#1\end{equation}}
\def\bes#1{\begin{equation*}\begin{split}#1\end{split}\end{equation*}}
\def\besn#1{\begin{equation}\begin{split}#1\end{split}\end{equation}}

\def\beqn#1\eeqn{\begin{align}#1\end{align}}
\def\beq#1\eeq{\begin{align*}#1\end{align*}}

\usepackage{graphicx}
\usepackage{latexsym}
\usepackage{amsmath,amsthm,amssymb,amscd}
\usepackage{epsf,amsmath}

\DeclareMathOperator*{\argsup}{argsup}
\def\E{{\IE}}

\renewcommand\section{\@startsection {section}{1}{\z@}%
{-3.5ex \@plus -1ex \@minus -.2ex}%
{1.3ex \@plus.2ex}%
{\center\small\sc\mathversion{bold}\MakeUppercase}}

\def\subsection#1{\@startsection {subsection}{2}{0pt}%
{-3.5ex \@plus -1ex \@minus -.2ex}%
{1ex \@plus.2ex}%
{\bf\mathversion{bold}}{#1}}

\def\subsubsection#1{\@startsection{subsubsection}{3}{0pt}%
{\medskipamount}%
{-10pt}%
{\normalsize\itshape}{\kern-2.2ex. #1.}}

\def\blfootnote{\xdef\@thefnmark{}\@footnotetext}

\makeatother

\begin{document}

\title{Limit theorems with rate of convergence under sublinear expectations}
\author{Xiao Fang$^*$, Shige Peng$^\dagger$, Qi-Man Shao$^*$, Yongsheng Song$^\ddagger$}
\date{\it The Chinese University of Hong Kong$^*$, Shandong University$^\dagger$, Chinese Academy of Sciences$^\ddagger$} 
\maketitle

\noindent{\bf Abstract:} Under the sublinear expectation $\mathbbm{E}[\cdot]:=\sup_{\theta\in \Theta} E_\theta[\cdot]$ for a given set of linear expectations $\{E_\theta: \theta\in \Theta\}$, we establish a new law of large numbers and a new central limit theorem with rate of convergence. We present some interesting special cases and discuss a related statistical inference problem. We also give an approximation and a representation of the $G$-normal distribution, which was used as the limit in Peng (2007)'s central limit theorem, in a probability space.

\medskip





\section{Introduction}

Let $\{P_\theta: \theta\in \Theta\}$ be a set of probability measures on a measurable space $(\Omega, \mathcal{F})$.
Let $E_\theta$ denote the expectation under $P_\theta$.
For a random variable $X: \Omega \to \mathbb{R}$ such that $E_\theta[X]$ exists for all $\theta \in \Theta$, we define its {\it sublinear} expectation as
\ben{\label{00}
\E[X]:=\sup_{\theta\in \Theta} E_\theta[X].
}
It is clear that the sublinear expectation \eq{00} satisfies the following: (i) {\it monotonicity} ($\E[X]\geq \E[Y]$ if $X\geq Y$), (ii) {\it constant preservation} ($\E[c]=c$ for $c\in \mathbb{R}$), (iii) {\it sub-additivity} ($\E[X+Y]\leq \E[X]+\E[Y]$), and (iv) {\it positive homogeneity} ($\E[\lambda X]=\lambda \E[X]$ for $\lambda\geq 0$).

From (iii), we have
\ben{\label{103}
\E[X]-\E[-Y]\leq \E[X+Y]\leq \E[X]+\E[Y].
}
In the special case where $Y$ does not have the mean uncertainty, that is, 
\ben{\label{33}
\E[Y]=-\E[-Y], \ \text{we have}\ \E[X+Y]=\E[X]+\E[Y].
}
From \eq{103} and (i), we have
\ben{\label{101}
\big|\E[X+Y]-\E[X]\big|\leq \E[|Y|].
}

Such a notion of sublinear expectation is often  used in situations where it is difficult or impossible to find the ``true'' probablity $P_{\hat{\theta}}$ among a set of uncertain probability models 
 $\{P_{\theta}\}_{\theta\in\Theta}$.
To the best of our knowledge, the definition \eq{00} of sublinear expectation first appeared in \cite{Hu81}, who called it the upper expectation. It was also called the upper prevision in the theory of imprecise probabilities. See, for example, \cite{Wa91}.  A type of nonlinear expectation adapted with a Brownian filtration, called $g$-expectation, was defined in \cite{p97}.  The sublinear situation of $g$-expectation 
was applied in \cite{Chen-Epstein02} to describe the investors' ambiguity aversions.  The notion of coherent risk measures introduced in \cite{ArDe99} is also a type of sublinear expectation. See also \cite{FoSc11}. The motivation for these related notions is to use the set of probability measures $\{P_\theta: \theta\in \Theta\}$ to model the uncertainty of probabilities and distributions in real data, and use the sublinear expectation $\E$ as a robust method to measure the risk loss $X$. We also refer to \cite{DPR2010} and  \cite{Pe10}  for more information on sublinear expectations, dynamical risk measures and general nonlinear expectations.

According to \cite{Pe07}, we say two random variables $X$ and $Y$ are {\it identically distributed}, denoted by $X\overset{d}{=}Y$, if
\be{
\E[\varphi(X)]=\E[\varphi(Y)]
}
for all bounded continuous functions $\varphi$.  $X\overset{d}{=}Y$ means that the distribution uncertainties of $X$ and $Y$ are the same.
There are different notions of independence under sublinear or nonlinear expectations. 
See, for example, \cite{Wa91}, \cite{Ma99} and \cite{MaMa05}.
We adopt the notion introduced by \cite{Pe10}
and say that a random vector $Y\in \mathbb{R}^n$ is {\it independent} of another random vector $X\in \mathbb{R}^m$ if
\be{
\E[\varphi(X, Y)]=\E\Big\{ \big\{\E[\varphi(x, Y)]\big\}_{x=X}\Big\}
}
for all bounded continuous functions $\varphi: \mathbb{R}^{m+n}\to \mathbb{R}$.  This independence often occurs in many situations 
where  the value of $X$ is realized before that of $Y$, but the distribution uncertainty of $Y$ does not change after this realization. 
A sequence of random variables $\{X_i\}_{i=1}^{\infty}$ is said to be {\it i.i.d.}\ if for each $i=1,2,\dots,$ $X_{i+1}$ is identically distributed as $X_1$ and independent of $(X_1,\dots, X_i)$. 
Under sublinear expectations, ``$Y$ is independent of $X$" does not imply automatically that ``$X$ is independent of $Y$".
Example 3.13 of \cite{Pe10} provides such an example.
In the special case that $\Theta$ is a singleton, \eq{00} reduces to the usual definition of expectation, and the definition of i.i.d.\  random variables reduces to that in the classical setting. 

\cite{Pe10} formulated a law of large numbers (LLN) under the sublinear expectation \eq{00} as follows. 
Let $\{X_i\}_{i=1}^\infty$ be an i.i.d.\  sequence of random variables with $\E[X_1]=\overline{\mu},$ $-\E[-X_1]=\underline{\mu}$, both being finite. 
By the definition of sublinear expectation \eq{00}, we have $\underline{\mu}\leq \overline{\mu}$.
Let $\overline{X}_n=(X_1+\dots+ X_n)/n$. Then, we have
\ben{\label{001}
\text{(LLN)}\quad \lim_{n\to \infty} \E[\varphi(\overline{X}_n)]\to \sup_{\underline{\mu}\leq y\leq \overline{\mu}} \varphi(y), \ \text{for} \ \varphi\in lip(\mathbb{R}),
}
where $lip(\mathbb{R})$ denotes the class of Lipschitz functions.
We refer to \eq{001} as the weak convergence of $\overline{X}_n$ to the {\it maximal distribution} with parameters $\underline{\mu}$ and $\overline{\mu}$.

By assuming further that $\underline{\mu}=\overline{\mu}=:\mu$ and
\be{
\E[(X_1-\mu)^2]=\overline{\sigma}^2,\quad -\E[-(X_1-\mu)^2]=\underline{\sigma}^2, \quad \E[|X_1|^3]< \infty,
}
\cite{Pe07} obtained a central limit theorem (CLT):
\ben{\label{002}
\text{(CLT)}\quad \lim_{n\to \infty} \E\big\{ \varphi[\sqrt{n}(\overline{X}_n-\mu)]\big\} \to u(1,0), \ \text{for} \ \varphi\in lip(\mathbb{R}),
}
where $\{u(t, x): (t,x)\in [0,\infty)\times \mathbb{R}\}$ is the unique viscosity solution to the following parabolic partial differential equation (PDE) defined on $[0,\infty)\times \mathbb{R}$:
\ben{\label{102}
\partial_t u -G(\partial^2_{xx} u)=0, \  u |_{t=0}=\varphi,
}
where $G=G_{\underline{\sigma}, \overline{\sigma}}(\alpha)$
is the following function parametrized by $\underline{\sigma}$ and $\overline{\sigma}$:
\be{
G(\alpha)=
\frac{1}{2}(\overline{\sigma}^2 \alpha^+ -\underline{\sigma}^2 \alpha^-), \ \alpha\in \mathbb{R}.
}
Here we denote $\alpha^+:=\max\{0, \alpha\}$ and $\alpha^-:=(-\alpha)^+$. 

We refer to \eq{002} as the weak convergence of $\sqrt{n}(\overline{X}_n-\mu)$ to the {\it $G$-normal distribution} with parameters $\underline{\sigma}^2$ and $\overline{\sigma}^2$.
We will denote the right-hand side of \eq{002} by $\mathcal{N}_G[\varphi]$ and suppress its dependence on $\underline{\sigma}^2$ and $\overline{\sigma}^2$ for the ease of notation. Recently, \cite{So17} obtained a convergence rate for Peng's CLT \eq{002},
which is of the order $O(1/n^{\alpha/2})$ with an unspecified $\alpha\in (0,1)$.

In the special case that $\varphi$ is a convex function, we can verify by the Gaussian integration by parts formula
and $G(\partial^2_{xx}u)=\frac{\overline{\sigma}^2}{2} \partial^2_{xx} u$ from the convexity of $\varphi$
that 
\ben{\label{107}
u(t,x)=E[\varphi(x+t^{1/2} \overline{\sigma} Z)]
}
is the solution to the PDE \eq{102}, where $Z$ is a standard Gaussian random variable.
Therefore, the limit in \eq{002} is a normal distribution.
The same conclusion holds for concave $\varphi$, except that $\overline{\sigma}$ is replaced by $\underline{\sigma}$.
The limit in \eq{002} is also normal for any $\varphi\in lip(\mathbb{R})$ if $\overline{\sigma}=\underline{\sigma}$.

Note that \eq{001} and \eq{002} reduce to classical LLN and CLT if $\Theta$ in \eq{00} is a singleton. In this case $\E$ is a linear expectation.

The goal of this paper is to obtain  convergence rates for the above  LLN and a new type of renormalized  CLT  with explicitly given bounds in the framework of sublinear expectations. 
For  the LLN, we prove that 
\be{
\E\big\{[(\overline{X}_n-\overline{\mu})^+]^2+[(\overline{X}_n-\underline{\mu})^{-}]^2\big\}\leq \frac{2[\overline{\sigma}^2 + (\overline{\mu}-\underline{\mu})^2]}{n},
}
where 
\be{
\overline{\sigma}^2:=\sup_{\theta\in \Theta}E_\theta\big\{[X_1-E_\theta(X_1)]^2 \big\}.
}
This upper bound provides us with a quantitative version of the fact that
for large $n$, the sample mean is sufficiently concentrated inside the interval $[\underline{\mu}, \overline{\mu}]$. 
We deduce this upper bound from a new law of large numbers, which may be of independent interest. 
We will discuss a related statistical inference problem under sublinear expectations.
We also discuss extensions to the multi-dimensional setting.

With respect to the CLT in \eq{002}, for the special case that $\varphi$ is a convex function, we prove that
\be{
\Big|\E\Big\{\varphi\big[\frac{\sqrt{n}}{\overline{\sigma}}(\overline{X}_n-\mu)\big] \Big\}-E[\varphi(Z)] \Big|
\leq \frac{\log n +1}{\sqrt{n}} \Big\{2+ \E\Big[\big|\frac{X_1-\mu}{\overline{\sigma}}\big|^3\Big]\Big\} ||\varphi'||,
}
where $Z$ is a standard Gaussian random variable and $||\cdot||$ denotes the supremum norm of a function.
A similar bound for $\varphi$ being a concave function is also obtained.
For the general case where the mean of $X_1$ is uncertain (that is, $\underline{\mu}\ne \overline{\mu}$) and $\varphi$ may not be convex or concave,
we formulate a new central limit theorem for
\be{
\sum_{i=1}^n\frac{X_i-\mu_i}{\sigma_i\sqrt{n}},
}
where $\mu_i$ equals $\overline{\mu}$ or $\underline{\mu}$ depending on previous $\{X_j: j<i\}$ and the solution to the heat equation, and $\sigma_i$ depends furthermore on the set of the possible first two moments of $X_1$. Our main tool for proving the rate of convergence for the CLT is a combination of Lindeberg's swapping argument and Stein's method. This approach was used by \cite{Ro17} for proving a martingale CLT.

The sublinear expectation \eq{00} is defined through a class of probability measures, and in general, cannot be represented in a single probability space. However, for the $G$-normal distribution, which was used as the limit in Peng's CLT \eq{002}, we can give an approximation and a representation in a probability space.

The rest of this paper is organized as follows. In Section 2, we present our results on the law of large numbers. Section 3 contains the results related to the CLT. A new representation of the $G$-normal distribution is derived in Section 4. 
Most of the proofs are deferred to Section 5.

\section{Law of large numbers}

In this section, we first provide a rate of convergence for Peng's law of large numbers, then discuss its implication on the statistical inference for uncertain distributions, and finally, we present a new law of large numbers with rates that may be of independent interest.

\subsection{Rate of convergence}

Let $\{X_i\}_{i=1}^\infty$ be an i.i.d.\  sequence of random variables under a sublinear expectation $\E$
such that
\be{
\E[X]=\sup_{\theta\in \Theta} E_\theta[X]
}
for a family of linear expectations $\{E_\theta: \theta\in \Theta\}$.
Suppose both
$\overline{\mu}=\E[X_1]\  \text{and}  \  \underline{\mu}=- \E [-X_1]$
are finite. Define
\ben{\label{120}
\overline{\sigma}^2:=\sup_{\theta\in \Theta}E_\theta\Big\{\big[X_1-E_\theta(X_1)\big]^2\Big\}.
}
If $\overline{\sigma}^2$ is finite, then we can control the expected deviation of the sample mean 
$\overline{X}_n=\sum_{i=1}^n X_i/n$
from the interval $[\underline{\mu}, \overline{\mu}]$.

\begin{theorem}\label{t7}
Under the above setting,
we have
\ben{\label{1001}
\E\Big\{\big[(\overline{X}_n-\overline{\mu})^+\big]^2+\big[(\overline{X}_n-\underline{\mu})^{-}\big]^2 \Big\}\leq \frac{2[\overline{\sigma}^2 + (\overline{\mu}-\underline{\mu})^2]}{n}.
}
\end{theorem}

\begin {remark}
We can rewrite \eq{1001} as
\be{
\E \big[d^2_{[\underline{\mu}, \overline{\mu}]} (\overline{X}_n)\big] \leq \frac{2[\overline{\sigma}^2 + (\overline{\mu}-\underline{\mu})^2]}{n},
}
where for $A\subset \mathbb{R}^d$ and $x\in \mathbb{R}^d$, $d_A(x):=\inf_{y\in A}|y-x|$.
Clearly, for any interval $I$ larger than $[\underline{\mu}, \overline{\mu}]$, i.e., $[\underline{\mu}, \overline{\mu}]\subset I$, the conclusion of Theorem \ref {t7} still holds for $d^2_{I}(\overline{X}_n)$.  In fact, $[\underline{\mu}, \overline{\mu}]$ is the smallest interval satisfying Theorem \ref {t7}.
According to (\ref{001}),   if $[\underline{\nu},\overline{\nu}]\nsupseteq [\underline{\mu},\overline{\mu}]$, then
\[\lim_{n\to \infty}\mathbb{E}\big[d_{[\underline{\nu},\overline{\nu}]}(\overline{X}_n)\big]=\sup_{x\in[\underline{\mu}, \overline{\mu}]}d_{[\underline{\nu},\overline{\nu}]}(x)>0.\] 
\end {remark}

\begin{remark} \eq{001} presents a law of large numbers under sublinear expectations where the convergence is in the distribution. In fact, if $\overline{\mu}>\underline{\mu}$, the convergence would not be in the strong sense: there does not exist a random variable $\eta$ such that
\besn{\label{se1}
\lim_{n\rightarrow\infty}\mathbb{E}\big[\big|\overline{X}_n-\eta\big|\big]=0.
}
Indeed, if (\ref{se1}) holds, then by \eq{001}, $\eta$ must be maximally distributed.  Set  $g(x)=\min\big\{\max\{x, \underline{\mu}\big\}, \overline{\mu}\}-\underline{\mu}$. On one hand, (\ref{se1}) implies that \[\lim_{n\rightarrow\infty}\mathbb{E}\big[-g(\overline{X}_n)(\eta-\underline{\mu}+1)\big]=\mathbb{E}\big[-g(\eta)(\eta-\underline{\mu}+1)\big]=0.\] On the other hand, as $\eta$ is independent of $g(S_n)$, we have
\begin{eqnarray*}\lim_{n\rightarrow\infty}\mathbb{E}\big[-g(\overline{X}_n)(\eta-\underline{\mu}+1)\big]&=&\lim_{n\rightarrow\infty}\mathbb{E}\big[g(\overline{X}_n)]\mathbb{E}[-(\eta-\underline{\mu}+1)\big]\\
&=&\mathbb{E}[g(\eta)]\mathbb{E}[-(\eta-\underline{\mu}+1)]=-(\overline{\mu}-\underline{\mu}).
\end{eqnarray*} This is a contradiction.
\end{remark}

Theorem \ref{t7} can be generalized to the multi-dimensional setting.

\begin{theorem}\label{t9}
Let $\{X_i\}_{i=1}^\infty$ be an i.i.d.\  sequence of $d$-dimensional random vectors 
under a sublinear expectation $\E=\sup_{\theta\in \Theta} E_\theta$.
Suppose that the convex hull of the closure of all the possible means $\{E_\theta [X_1]: \theta\in \Theta\}$ is  a  bounded convex polytope $\mathcal{P}$ with $m$ vertices. We have
\be{
\E \big[d^2_{\mathcal{P}} (\overline{X}_n)\big]\leq \frac{m\big\{ \sup_{\theta\in \Theta} E_\theta\big[|X_1-E_\theta[X_1]|^2\big]+diam^2(\mathcal{P})\big\}}{n},
}
where $\overline{X}_n=\sum_{i=1}^n X_i/n$, $|\cdot|$ denotes the Euclidean norm, and $diam(\mathcal{P})$ denotes the diameter of the polytope.
\end{theorem}

Theorem \ref{t9} reduces to Theorem \ref{t7} in the one-dimensional case by regarding $[\underline{\mu}, \overline{\mu}]$ as a polytope with $m=2$ vertices.
Theorem \ref{t9} follows from a new law of large numbers stated in Section 2.3, which may be of independent interest.

\begin {remark} \label {r3} Based on  Theorem \ref{t9}, we can also give a convergence rate of $\E [d_{\mathcal{P}} (\overline{X}_n)]$ when $\mathcal{P}$ is a general convex set in $\mathbb{R}^d$ with a regular boundary.  For example,  if $\mathcal{P}$ is a disk of radius $R$ in a plane, we have (proof deferred to Section 5.1)
\[\mathbb{E}[d_{\mathcal{P}}(\overline{X}_n)]\le \frac{\big(7\pi^2R+\sqrt{\overline{\sigma}^2+16R^2}\big)}{n^{\frac{2}{5}}}. \]

\end {remark}

\subsection{Statistical inference for uncertain distributions}

The upper bound in Theorem \ref{t7} provides us with a quantitative version of the fact that
for large $n$, the sample mean is sufficiently concentrated inside the interval $[\underline{\mu}, \overline{\mu}]$.
This is related to the estimation of $\underline{\mu}$ and $\overline{\mu}$ described below.

Given an i.i.d.\  sequence of random variables $X_1, \dots, X_N$ under linear expectations, the usual estimator for their mean is
\be{
\hat{\mu}=\frac{X_1+\dots +X_N}{N}.
}
Here, we consider a statistical estimation under sublinear expectations.

Let $X_1,\dots, X_N$ be an i.i.d.\  sequence of random variables under a sublinear expectation $\E$
such that
\be{
\E[X]=\sup_{\theta\in \Theta} E_\theta[X]
}
for a family of linear expectations $\{E_\theta: \theta\in \Theta\}$.
Suppose that $N=nk$ and the data are expressed as follows:
\be{
\begin{bmatrix}
X_{11} & \dots & X_{1n} \\
\vdots & \vdots & \vdots \\
X_{k1} & \dots & X_{kn}
\end{bmatrix}.
}
\cite{JiPe16} proposed to estimate the lower mean $\underline{\mu}$ and the upper mean $\overline{\mu}$ of $X_1$ by
\be{
\hat{\underline{\mu}}:=\min_{1\leq j\leq k} \frac{\sum_{i=1}^n X_{ji}}{n}
}
and
\be{
\hat{\overline{\mu}}:=\max_{1\leq j\leq k} \frac{\sum_{i=1}^n X_{ji}}{n},
}
respectively.
Applying Theorem \ref{t7} and the union bound, we have the following result.

\begin{proposition}\label{p1}
Suppose $\E[X_1^2]<\infty$. We have
\be{
\E\big\{\big[(\hat{\overline{\mu}}-\overline{\mu})^+\big]^2\big\}\leq \frac{Ck}{n},\quad \text{and} \quad
\E\big\{\big[(\hat{\underline{\mu}}-\underline{\mu})^-\big]^2\big\}\leq \frac{Ck}{n},
}
where $C$ is a constant depending only on $\underline{\mu}, \overline{\mu}$ and $\overline{\sigma}^2$ in \eq{120}.
\end{proposition}
\begin{proof}
Define
\be{
Y_j:=\frac{\sum_{i=1}^n X_{ji}}{n}.
}
We have, by the union bound and Theorem \ref{t7},
\bes{
&\E\big\{\big[(\hat{\overline{\mu}}-\overline{\mu})^+\big]^2\big\}=\E\big\{\big[(\max_{1\leq j\leq k} Y_j-\overline{\mu})^+\big]^2\big\}\\
= &\E\big\{\max_{1\leq j\leq k}\big[( Y_j-\overline{\mu})^+\big]^2\big\}\leq \E\big\{ \sum_{j=1}^k\big[ (Y_j-\overline{\mu})^+\big]^2\big\}
\leq \frac{Ck}{n}.
}
The second inequality follows from the same argument.
\end{proof}

Proposition \ref{p1} ensures that as $n\to \infty$ and $k=o(n)$, the estimators by \cite{JiPe16} are sufficiently concentrated inside $[\underline{\mu}, \overline{\mu}]$.


\subsection{A new law of large numbers}

We first formulate a new law of large numbers for the one-dimensional case.
\begin{theorem}\label{t2}
Let $\{X_i\}_{i=1}^\infty$ be an i.i.d.\  sequence of random variables under a sublinear expectation $\E$ such that
\bes{
\E[X]=\sup_{\theta\in \Theta} E_\theta[X]
}
for a family of linear expectations $\{E_\theta: \theta\in \Theta\}$.
Suppose that $\E[ X_1^2]<\infty$.
Denote
\be{
\overline{\mu}:=\E[X_1],\ \underline{\mu}:=- \E [-X_1].
}
Then, for $\varphi$ differentiable such that $\varphi'\in lip(\mathbb{R})$, we have
\be{
\Big|\E\Big\{\varphi\big[\frac{\sum_{i=1}^n (X_i-\mu_i)}{n}\big]\Big\}-\varphi(0)\Big|\leq \frac{C_0||\varphi''||}{n}
}
where 
\be{
\mu_i=
\begin{cases}
\overline{\mu},\ \text{if}\ \varphi'\big[\frac{\sum_{j=1}^{i-1} (X_j-\mu_j)}{n}\big]\geq 0,\\
\underline{\mu},\ \text{if}\ \varphi'\big[\frac{\sum_{j=1}^{i-1} (X_j-\mu_j)}{n}\big]< 0,
\end{cases}
}
and
\be{
C_0=\frac{1}{2}[\overline{\sigma}^2 +(\overline{\mu}-\underline{\mu})^2], \quad \overline{\sigma}^2=\sup_{\theta\in \Theta}E_\theta\Big\{\big[X_1-E_\theta(X_1) \big]^2 \Big\}.
}
\end{theorem}

Theorem \ref{t2} is a direct consequence of the following multivariate version, which will be proved in Section 5.1.
\begin{theorem}\label{t4}
Let $X_1, X_2, \dots$ be an i.i.d.\  sequence of $d$-dimensional random vectors under a sublinear expectation $\E$ such that
\be{
\E[X]=\sup_{\theta\in \Theta} E_\theta[X]
}
for a family of linear expectations $\{E_\theta: \theta\in \Theta\}$.
Let
\be{
\mathcal{M}_1:=\{E_\theta[X_1]: \theta\in \Theta\}
}
be all possible means of $X_1$.
Let $\mathcal{P}$ be the convex hull of the closure of $\mathcal{M}_1$.
We have, for $\varphi: \mathbb{R}^d\to \mathbb{R}$ differentiable such that the gradient $D\varphi: \mathbb{R}^d\to \mathbb{R}^d$ is a Lipschitz function,
\be{
\Big|\E\Big\{\varphi\big[\frac{\sum_{i=1}^n (X_i-\mu_i)}{n}\big]\Big\}-\varphi(0) \Big|\leq \frac{\lambda_* \big\{\sup_{\theta\in \Theta} E_\theta\big[|X_1-E_\theta[X_1]|^2\big]+diam^2(\mathcal{P})\big\}}{2n},
}
where
$\mu_i:=\argsup_{\mu\in \mathcal{P}}\big\{\mu\cdot D\varphi\big[\frac{\sum_{j=1}^{i-1}(X_j-\mu_j)}{n}\big]\big\}$ (if the $\argsup$ is not unique, choose any value),
$\lambda_*$ is the supremum norm of the operator norm of the Hessian $D^2\varphi$, and $diam(\mathcal{P})$ denotes the diameter of $\mathcal{P}$.
\end{theorem}

\section{Central limit theorem with rate of convergence}

As explained in the Introduction, in the special case where $\varphi$ is a convex or concave test function, the limit in Peng's CLT in \eq{002} is a usual normal distribution. We first provide a rate of convergence for this special case. Moreover, unlike in \eq{002}, we do not need to impose the {\it identically distributed} assumption.

\begin{theorem}\label{t6}
Suppose $X_1,\dots, X_n$ are independent under a sublinear expectation $\E$ with
\be{
\E[X_i]=-\E[-X_i]=\mu, \ \E[(X_i-\mu)^2]=\overline{\sigma}_i^2, -\E[-(X_i-\mu)^2]=\underline\sigma_i^2.
}
Let 
\be{
\sum_{i=1}^n \overline\sigma_i^2=\overline{B}_n^2.
}
For convex test functions $\varphi(\cdot)\in lip(\mathbb{R})$, we have
\be{
\Big|\E \Big\{\varphi\big[\sum_{i=1}^n (X_i-\mu)/\overline{B}_n\big] \Big\}-E[\varphi(Z)] \Big|\leq  \frac{||\varphi'||}{\overline{B}_n} \sum_{i=1}^n  \frac{2\overline{\sigma}_i^3+\E [|X_i-\mu|^3]}{\sum_{j=i}^n \overline{\sigma}_j^2},
}
where $Z$ is a standard Gaussian random variable.
For concave functions $\varphi$, if we let
\be{
\sum_{i=1}^n \underline\sigma_i^2=\underline{B}_n^2, 
}
then
\be{
\Big|\E \Big\{\varphi\big[\sum_{i=1}^n (X_i-\mu)/\underline{B}_n\big] \Big\}-E[\varphi(Z)] \Big|\leq  \frac{||\varphi'||}{\underline{B}_n} \sum_{i=1}^n  \frac{2\underline{\sigma}_i^2\overline{\sigma}_i+\E [|X_i-\mu|^3]}{\sum_{j=i}^n \underline{\sigma}_j^2}.
}
\end{theorem}

The proof of Theorem \ref{t6} follows from a similar and simpler proof of Theorem \ref{t1} below and is deferred to Section 5.2.
Theorem \ref{t6} has the following corollary if the $X_i$'s are assumed to be i.i.d.\ 
\begin{corollary}\label{t0}
Under the conditions of Theorem \ref{t6},
suppose further that $X_1,\dots, X_n$ are i.i.d., and denote
\be{
\overline{\sigma}^2:=\E[(X_1-\mu)^2],\quad \underline{\sigma}^2:=-\E[-(X_1-\mu)^2].
}
Then, for a convex test function $\varphi\in lip(\mathbb{R})$, we have
\be{
\Big|\E\Big\{\varphi\Big[\frac{\sum_{i=1}^n(X_i-\mu)}{\overline{\sigma}\sqrt{n}}\Big]\Big\}-E[\varphi(Z)] \Big|
\leq \frac{\log n +1}{\sqrt{n}} \Big(2+ \E\Big[\Big|\frac{X_1-\mu}{\overline{\sigma}}\Big|^3\Big]\Big) ||\varphi'||,
}
where $Z$ is a standard Gaussian random variable.
If $\varphi$ is concave,
then we have
\be{
\Big|\E\Big\{\varphi\Big[\frac{\sum_{i=1}^n(X_i-\mu)}{\underline{\sigma}\sqrt{n}}\Big]\Big\}-E[\varphi(Z)] \Big|
\leq \frac{\log n +1}{\sqrt{n}} \Big(\frac{2\overline{\sigma}}{\underline{\sigma}}+ \E\Big[\Big|\frac{X_1-\mu}{\underline{\sigma}}\Big|^3\Big]\Big) ||\varphi'|| .
}
\end{corollary}
\begin{proof}
Corollary \ref{t0} follows directly from Theorem \ref{t6} by 
\be{
\overline{B}_n^2=n \overline{\sigma}_i^2\ \text{and}\ \underline{B}_n^2=n \underline{\sigma}_i^2\ \text{for all}\ i=1,\dots, n,
}
and the fact that
$1+\dots+\frac{1}{n}\leq \log n+1.$
\end{proof}

For the general case where the mean of $X_1$ is uncertain (that is, $\underline{\mu}\ne \overline{\mu}$) and $\varphi$ may not be convex or concave,
we formulate a new CLT for
\be{
\sum_{i=1}^n\frac{X_i-\mu_i}{\sigma_i\sqrt{n}},
}
where $\mu_i$ equals $\overline{\mu}$ or $\underline{\mu}$ depending on previous $\{X_j: j<i\}$ and the solution to the heat equation, and $\sigma_i$ depends furthermore on the set of the possible first two moments of $X_1$.
As above, let $\{X_i\}_{i=1}^\infty$ be an i.i.d.\  sequence of random variables under a sublinear expectation $\E$ such that
\bes{
\E[X]=\sup_{\theta\in \Theta} E_\theta[X]
}
for a family of linear expectations $\{E_\theta: \theta\in \Theta\}$.
Suppose that $\E[ |X_1|^3]<\infty$.
Define
\ben{\label{12}
\overline{\mu}:= \E(X_1),\ \underline{\mu}:=- \E [-X_1],
}
and for each possible mean $\mu$ of $X_1$, define
\ben{\label{121}
\overline{\sigma}_\mu^2:=\sup_{\theta\in \Theta: E_\theta(X_1)=\mu} E_\theta[(X_1-\mu)^2], \
\underline{\sigma}_\mu^2:=\inf_{\theta\in \Theta: E_\theta(X_1)=\mu} E_\theta[(X_1-\mu)^2].
}
We impose the following assumption:

\noindent {\bf Assumption A}. \textit{Regarded as functions of $\mu$, $\overline{\sigma}_\mu^2$ and $\underline{\sigma}_\mu^2$ are continuous at, or can be continuously extended to, $\mu=\overline{\mu}$ and $\mu=\underline{\mu}$.}

Denote
\besn{\label{122}
&\overline{\sigma}_{\overline{\mu}}^2:= \lim_{\mu\to \overline{\mu}^-} \overline{\sigma}_\mu^2,
\quad \underline{\sigma}_{\overline{\mu}}^2:= \lim_{\mu\to \overline{\mu}^-} \underline{\sigma}_\mu^2,
\\
&\overline{\sigma}_{\underline{\mu}}^2:= \lim_{\mu\to \underline{\mu}^-} \overline{\sigma}_\mu^2,
\quad \underline{\sigma}_{\underline{\mu}}^2:= \lim_{\mu\to \underline{\mu}^-} \underline{\sigma}_\mu^2.
}
There is no conflict of notation between \eq{121} and \eq{122} by Assumption A.
We assume further that

\noindent {\bf Assumption B}. \textit{All the four quantities in \eq{122} are positive.}

Let
\besn{\label{18}
\mathcal{M}_2=\big\{\big(E_\theta [X_1], E_\theta [X_1-E_\theta (X_1)]^2\big): \theta\in \Theta\big\}
}
be the set of all possible pairs of mean and variance of $X_1$.
Define
\ben{\label{51}
\sigma_0^2:=\min_{\mu_i=\underline{\mu}\ \text{or}\ \overline{\mu}} \{ \overline{\sigma}^2_{\mu_i} \wedge \inf_{(\mu,\sigma^2)\in \mathcal{M}_2} [\sigma^2+(\mu-\mu_i)^2] \},
}
\be{
\overline{\sigma}^2:=\sup_{\theta\in \Theta}E_\theta\Big\{\big[X_1-E_\theta(X_1) \big]^2 \Big\},
}
and
\be{
\overline{\gamma}:=\sup_{\theta\in \Theta}E_\theta \Big[|X_1-E_\theta[X_1]|^3 \Big].
}
On the basis of Assumptions A and B, we have $\sigma_0^2>0$.
We have the following theorem. 

\begin{theorem}\label{t1}
Under the above setting, we have the following CLT: for each $\varphi\in lip(\mathbb{R})$,
\ben{\label{15}
\Big| \E \big[\varphi(\frac{1}{\sqrt{n}} \sum_{i=1}^n  \frac{X_i-\mu_i}{\sigma_i})\big]-E[\varphi(Z)] \Big|\leq \frac{C_1 (\log n+1)}{\sqrt{n}} ||\varphi'||.
}
In \eq{15},
\be{
C_1=2+ \frac{5[\overline{\sigma}+(\overline{\mu}-\underline{\mu})]}{\sigma_0}
+\frac{4[\overline{\gamma}+(\overline{\mu}-\underline{\mu})^3]}{\sigma_0^3},
}
$Z$ is a standard Gaussian random variable, with
\be{
t_i=\frac{n-i}{n}, \  W_i=\frac{1}{\sqrt{n}}\sum_{j=1}^{i} \frac{X_j-\mu_j}{\sigma_j},
}
$\mu_i=\mu_i((X_j, \mu_j, \sigma_j): j<i)$ are defined as
\ben{\label{13}
\mu_i=
\begin{cases}
\overline{\mu} & \text{if}\ \partial_x V_{i-1}\geq 0, \\
\underline{\mu} & \text{if}\ \partial_x V_{i-1}< 0,
\end{cases}
}
$\sigma_i=\sigma_i((X_j, \mu_j, \sigma_j): j<i, \mu_i)$ are defined as
\ben{\label{14}
\sigma_i=
\begin{cases}
\inf\big\{b: b\geq \overline{\sigma}_{\mu_i},  \sup_{(\mu, \sigma^2)\in \mathcal{M}_2} [f_{i-1, b}(\mu, \sigma^2)]=0\big\}
& \text{if}\ \partial^2_{xx} V_{i-1}\geq 0,\\
\sup\big\{b: 0<b\leq \underline{\sigma}_{\mu_i},  \sup_{(\mu, \sigma^2)\in \mathcal{M}_2} [f_{i-1, b}(\mu, \sigma^2) ]=0\big\}
& \text{if}\
\partial^2_{xx} V_{i-1}< 0,
\end{cases}
}
where
\ben{\label{17}
f_{i-1, b}(\mu, \sigma^2)=\big[\frac{\mu-\mu_i}{b}\big]\frac{\partial_x V_{i-1}}{\sqrt{n}} +\big[\frac{\sigma^2+(\mu-\mu_i)^2}{b^2}-1\big] \frac{\partial^2_{xx}V_{i-1}}{n},
}
$V_{i-1}:= V(t_{i-1}, W_{i-1})$ and
$V(\cdot, \cdot)$ is the solution to the heat equation
\be{
\partial_t V (t, x)= \frac{1}{2} \partial^2_{xx} V(t, x), \ V(0,x)=\varphi(x).
}
\end{theorem}

The proof of Theorem \ref{t1} is deferred to Section 5.2.

\begin{remark}
From the definition of $\mu_i$ in \eq{13}, the first term of $f_{i-1, b}(\mu, \sigma^2)$ in \eq{17} is $\leq 0$ for $(\mu, \sigma^2)\in \mathcal{M}_2$ in \eq{18}.
It is straightforward to show, by checking the values of the supremum in \eq{14} at the boundary points below and by the fact that
$\sup_{(\mu, \sigma^2)\in \mathcal{M}_2} [f_{i-1, b}(\mu, \sigma^2) ]$ is continuous for $b$ in a compact set in $(0,\infty)$, that in \eq{14}, if $\partial^2_{xx} V_{i-1}\geq 0$, then
\be{
\overline{\sigma}_{\mu_i}^2\leq \sigma_i^2\leq \sup_{(\mu,\sigma^2)\in \mathcal{M}_2}\big[\sigma^2+(\mu-\mu_i)^2\big],
}
and if $\partial^2_{xx} V_{i-1}< 0$, then
\be{
\inf_{(\mu,\sigma^2)\in \mathcal{M}_2}\big[\sigma^2+(\mu-\mu_i)^2\big]\leq \sigma_i^2\leq \underline{\sigma}_{\mu_i}^2.
}
Therefore, $\sigma_i^2$ is well-defined and is bounded below by $\sigma_0^2$ in \eq{51}.
\end{remark}


\begin{remark}
In Theorem \ref{t1}, if we assume that $\overline{\mu}=\underline{\mu}=:\mu$, then it is easy to check that
\be{
\sigma_i=
\begin{cases}
\overline{\sigma}_{\mu}
& \text{if}\ \partial^2_{xx} V_{i-1}\geq 0,\\
\underline{\sigma}_{\mu}
& \text{if}\
\partial^2_{xx} V_{i-1}< 0.
\end{cases}
}
If we assume further that $\varphi$ is a convex (concave resp.) function
and hence $V(t,\cdot)=E\varphi(\cdot+\sqrt{t}Z)$ is convex (concave resp.),
then $\sigma_i$ is further reduced to $\overline{\sigma}_{\mu}$ ($\underline{\sigma}_{\mu} $ resp.).
In this special case, 
Theorem \ref{t1} reduces to Corollary \ref{t0} except for the constant.
\end{remark}

\section{Representation of $G$-normal distribution}

Under  the sublinear expectation, the $G$-normal distribution $\mathcal{N}_G$ plays the same role as the classical normal distribution does in a probability space (cf. \eq{002}). 
However,   since $\mathcal{N}_G$ is linked with a fully  nonlinear PDE, which is called $G$-heat equation, generally we cannot give an explicit expression for $\mathcal{N}_G[\varphi]$ like the linear case. So it would be important to give a representation or  approximation for $\mathcal{N}_G[\varphi] $ using random variables or processes in a probability space.

Theorem \ref {t1} shows that under a certain normalization, the partial sum of  i.i.d random variables in a sublinear expectation space converges to the standard normal distribution.  Motivated by this, in this section, we give an approximation of the $G$-normal distribution by using a suitably normalized  partial sum of  i.i.d.\ random variables in a probability space. Moreover, the continuous-time counterpart provides a representation of the $G$-normal distribution using (non-time-homogeneous) SDEs. This refines a result given in \cite {DHP11}, Proposition 49, which implies that the $G$-normal distribution can be represented by It\^o integrals with respect to a Brownian motion.

\subsection {Approximation of $G$-normal distribution}
Let $X_1, X_2, \cdots$ be a sequence of i.i.d random variables with $E[X_1]=0$ and $E[X_1^2]=1$ in a probability space $(\Omega,\mathcal{F}, P)$. Suppose further that $E[|X_1|^3]<\infty$.

Denote by $\Sigma^{\mathbb{N}}_G$ the collection of all the sequences of measurable functions $\{\sigma_i\}_{i=1}^{\infty}$ with $\sigma_i: \mathbb{R}\rightarrow [\underline{\sigma},\overline{\sigma}]$ for any $i\in\mathbb{N}$.
Fix $n\in\mathbb{N}$.  For a mapping $\sigma\in\Sigma^{\mathbb{N}}_G$, set  $W^\sigma_{0,n}=0$, and, for $1\le i\le n$, set
\begin {eqnarray} \label {D2-SDE}
W^\sigma_{i,n}=W^\sigma_{i-1,n}+\sigma_{i}(W^\sigma_{i-1,n})\frac{X_i}{\sqrt{n}}.
\end {eqnarray}
Thus, we have $W^\sigma_{i,n}=\frac{1}{\sqrt{n}}\sum\limits_{k=1}^i\Big(\sigma_{k}(W_{k-1,n}^\sigma) X_k\Big)=:\frac{1}{\sqrt{n}}\sum_{k=1}^i X^{\sigma}_{k,n}$. Write $W^\sigma_n=W^\sigma_{n,n}$ for simplicity.
\begin{theorem} \label {CLT}   For any $\varphi\in lip(\mathbb{R})$, we have
\[\Big|\sup\limits_{\sigma\in\Sigma^{\mathbb{N}}_G}E[\varphi(W_n^\sigma)]-\mathcal{N}_G[\varphi]\Big|\le C_{\alpha,G}\bar{\sigma}^{2+\alpha}\|\varphi'\|\frac{E[|X_1|^{2+\alpha}]}{n^{\frac{\alpha}{2}}},\] where $\alpha\in (0,1)$, and $C_{\alpha,G}>0$ are constants depending on $\underline{\sigma}$ and $\overline{\sigma}$.
\end{theorem}

The proof of Theorem \ref{CLT} is deferred to Section 5.3. 
We first express $\sup_{\sigma\in\Sigma^{\mathbb{N}}_G}E[\varphi(W_n^\sigma)]$ as a sublinear expectation of a sum of i.i.d.\ random variables. 
The theorem then follows from the error bound by \cite{So17} for \cite{Pe07}'s CLT.

\subsection {Representation of $G$-normal distribution}
Roughly speaking, the continuous-time form of Eq. (\ref{D2-SDE}) is
\begin {eqnarray} \label {phi-SDE}
dW^\sigma_t=\sigma(t,W^\sigma_t)dB_t, \ t\in(0,1],
\end {eqnarray} where $B$ is a standard Brownian motion in a filtered probability space $(\Omega,\mathcal{F}, \mathbb{F}, P)$.

Denote as $\Sigma_G$ the collection of all smooth functions $\sigma: [0,1]\times\mathbb{R}\rightarrow [\underline{\sigma},\overline{\sigma}]$ with
\[\sup\limits_{(t,x)\in[0,1]\times\mathbb{R}}|\partial_x\sigma(t,x)|<\infty.\]

For $\sigma\in\Sigma_G$, we consider the following stochastic differential equation SDE (\ref{phi-SDE}) with the initial value $x$:
\begin {eqnarray} \label {sigma-SDE}
\begin{split}
dW^{\sigma,x}_t&= \sigma(t,W^{\sigma,x}_t)dB_t,\ t\in(0,1],\\
W^{\sigma,x}_0&= x.
\end{split}
\end {eqnarray}
We write $W^{\sigma}$ for $W^{\sigma,0}$. Denote $\Theta_G:=\{P\circ (W_1^\sigma)^{-1}| \ \sigma\in\Sigma_G\}$. For a function $\sigma: [0,1]\times\mathbb{R}\rightarrow \mathbb{R}$, set $\widetilde{\sigma}(t,x)=\sigma(1-t,x)$.
\begin{theorem} \label{t10} For any $\varphi\in lip(\mathbb{R})$, we have
\[\mathcal{N}_G[\varphi]=\sup_{\mu\in\Theta_G}\mu[\varphi].\]
\end{theorem}

\begin {remark} Note that in the above representation, we need to use non-time-homogeneous SDEs. If we only consider time-homogeneous SDEs, the representation will be strictly smaller than the $G$-normal distribution.
\end {remark}

\section{Proofs}

\subsection{Proofs in Section 2}

In this subsection, we first prove Theorem \ref{t4} and then use it to prove Theorem \ref{t9}. Finally, we provide a simple explanation for Remark \ref{r3}.

\begin{proof}[Proof of Theorem \ref{t4}]
Denote
\be{
Y_0=0,\quad Y_k:=\sum_{i=1}^k \xi_i :=\sum_{i=1}^k\frac{ (X_i-\mu_i)}{n},
}
and denote
\be{
Y_{[k]}:=\{Y_1,\dots, Y_k\}.
}
For arbitrary random vectors $X$ and $Y$, denote
\be{
\E^X[\varphi(X, Y)]:=\big\{ \E[\varphi(x,Y)] \big\}_{x=X}.
}
We will prove the following claim.

\begin{claim}\label{claim1}
For any $k=1,\dots, n$, we have
\be{
\Big| \E^{Y_{[k-1]}} [\varphi(Y_k)-\varphi(Y_{k-1})] \Big|\leq \frac{\lambda_* \big\{ \sup_{\theta\in \Theta} E_\theta \big[|X_1-E_\theta[X_1]|^2 \big]+diam^2(\mathcal{P})\big\}}{2n^2}
}
\end{claim}
Using telescoping sum and the independence assumption and applying Claim \ref{claim1} recursively from $k=n$ to $k=1$, we have
\bes{
&\E[\varphi(Y_n)]-\varphi(0)\\
=&\E\Big\{ \sum_{k=1}^n [\varphi(Y_k)-\varphi(Y_{k-1})]  \Big\}\\
=&\E \E^{Y_{[n-1]}} \Big\{ \sum_{k=1}^n [\varphi(Y_k)-\varphi(Y_{k-1})]  \Big\}\\
=&\E\Big\{  \sum_{k=1}^{n-1} [\varphi(Y_k)-\varphi(Y_{k-1})]  + \E^{Y_{[n-1]}} [\varphi(Y_n)-\varphi(Y_{n-1})]  \Big\}\\
\leq & \E\Big\{  \sum_{k=1}^{n-1} [\varphi(Y_k)-\varphi(Y_{k-1})]   \Big\} + \frac{\lambda_* \big\{ \sup_{\theta\in \Theta} E_\theta \big[|X_1-E_\theta[X_1]|^2 \big]+diam^2(\mathcal{P})\big\}}{2n^2}\\
\leq & \dots \\
\leq &\frac{\lambda_* \big\{ \sup_{\theta\in \Theta} E_\theta \big[|X_1-E_\theta[X_1]|^2 \big]+diam^2(\mathcal{P})\big\}}{2n}.
}
The lower bound is proved by changing $\leq$ to $\geq$ and changing $+$ to $-$ for the error terms. 
Therefore, we obtain Theorem \ref{t4}, subject to Claim \ref{claim1}.

To prove Claim \ref{claim1}, we first write
\bes{
&\E^{Y_{[k-1]}} \big[\varphi(Y_k)-\varphi(Y_{k-1})\big]\\
=&\E^{Y_{k-1}} \big[\varphi(Y_k)-\varphi(Y_{k-1})\big]\\
=&\E^{Y_{k-1}}\Big [\xi_k \cdot D\varphi(Y_{k-1})+\int_0^1 \int_0^1  \xi_k^T D^2\varphi(Y_{k-1}+\alpha \beta \xi_k) \xi_k \alpha d\alpha d\beta\Big].
}
By the property \eq{101} of the sublinear expectation and the definition of $\lambda_*$, we have
\bes{
&\Big|\E^{Y_{[k-1]}} \big[\varphi(Y_k)-\varphi(Y_{k-1})\big]-  \E^{Y_{k-1}} \big[\xi_k \cdot D\varphi(Y_{k-1})\big]    \Big|\\
\leq & \E^{Y_{k-1}}\Big| \int_0^1 \int_0^1  \xi_k^T D^2\varphi(Y_{k-1}+\alpha \beta \xi_k) \xi_k \alpha d\alpha d\beta   \Big|\\
\leq &  \frac{1}{2} \lambda_*  \E\big[|\xi_k|^2\big].
}
Note that
\bes{
\E\big[|X_k-\mu_k|^2\big]=&\sup_{\theta\in \Theta} E_\theta \big[|X_k-\mu_k|^2\big]  \\
=&\sup_{\theta\in \Theta}\big\{ E_\theta\big[|X_k-E_\theta(X_k)|^2\big]+ \big|E_\theta[X_k]-\mu_k\big|^2\big\}\\
\leq & \sup_{\theta\in \Theta} E_\theta\big[|X_1-\mu_\theta|^2\big]+diam^2(\mathcal{P}).
}
Hence, 
\besn{\label{claim1-1}
&\Big|\E^{Y_{[k-1]}} \big[\varphi(Y_k)-\varphi(Y_{k-1})\big]-  \E^{Y_{k-1}} \big[\xi_k \cdot D\varphi(Y_{k-1})\big]    \Big|\\
\leq & \frac{\lambda_* \big\{ \sup_{\theta\in \Theta} E_\theta \big[|X_1-E_\theta[X_1]|^2 \big]+diam^2(\mathcal{P})\big\}}{2n^2}.
}
By the definition of sublinear expectation,
\be{
\E^{Y_{k-1}}\big[X_k \cdot D\varphi(Y_{k-1})\big]=\sup_{\mu\in \mathcal{M}_1} \mu \cdot  D\varphi(Y_{k-1}).
}
As $\mathcal{M}_1\subset \mathcal{P}$, it is clear that
\be{
\sup_{\mu\in \mathcal{M}_1} \mu \cdot  D\varphi(Y_{k-1})\leq \sup_{\mu\in \mathcal{P}} \mu \cdot  D\varphi(Y_{k-1}).
}
On the other hand, for $\lambda_1, \lambda_2\geq 0$ such that $\lambda_1+\lambda_2=1$ and $\mu_1, \mu_2\in \overline{\mathcal{M}}_1$, the closure of $\mathcal{M}_1$,
\be{
(\lambda_1 \mu_1 +\lambda_2 \mu_2)\cdot D\varphi(Y_{k-1}) \leq \sup_{\mu\in \overline{\mathcal{M}}_1} \mu\cdot D\varphi(Y_{k-1})=\sup_{\mu\in \mathcal{M}_1} \mu \cdot  D\varphi(Y_{k-1}).
}
Therefore,
\be{
\E^{Y_{k-1}}\big[X_k \cdot D\varphi(Y_{k-1})\big]=\sup_{\mu\in \mathcal{M}_1} \mu \cdot  D\varphi(Y_{k-1})=\sup_{\mu\in \mathcal{P}} \mu \cdot  D\varphi(Y_{k-1}),
}
and by the choice of $\mu_k$, we have
\be{
\E^{Y_{k-1}}\big[\xi_k \cdot D\varphi(Y_{k-1})\big]=0.
}
This, together with \eq{claim1-1}, proves Claim \ref{claim1}.

\end{proof}

\begin{proof}[Proof of Theorem \ref{t9}]
Here, $\mathcal{P}$ is a bounded convex polytope with $m$ vertices.
Denote the set of vertices by $\mathcal{V}$.
For each vertex $v\in \mathcal{V}$, define
\be{
T_v=\{w\in \mathbb{R}^d: w-v=c(u-v)\ \text{for some}\ u\in \mathcal{P}\ \text{and}\ c\geq 0\}.
}
It is clear that $\mathcal{P}=\cap_{v\in \mathcal{V}} T_v$ where the intersection is over all the $m$ vertices.
(Just to clarify the definitions, consider, for example, $d=1$ and $\mathcal{P}=[\underline{\mu},\overline{\mu}]$. It has two vertices $\mathcal{V}=\{\underline{\mu},\overline{\mu}\}$. Thus, we have $T_{\underline{\mu}}=[\underline{\mu},\infty)$, $T_{\overline{\mu}}=(-\infty, \overline{\mu}]$ and $\mathcal{P}=T_{\underline{\mu}} \cap T_{\overline{\mu}}$.)
We will prove that
\ben{\label{c3}
d_{T_v}^2\big(\frac{\sum_{i=1}^n X_i}{n}\big)\leq \frac{\sup_{\theta\in \Theta} E_\theta\big[|X_1-E_\theta(X_1)|^2\big]+diam^2(\mathcal{P})}{n}
}
and hence
\be{
d_{\mathcal{P}}^2\big(\frac{\sum_{i=1}^n X_i}{n}\big)\leq\sum_{v\in \mathcal{V}} d_{T_v}^2\big(\frac{\sum_{i=1}^n X_i}{n}\big) \leq \frac{m\big\{\sup_{\theta\in \Theta} E_\theta\big[|X_1-E_\theta(X_1)|^2\big]+diam^2(\mathcal{P})\big\}}{n}.
}
To prove \eq{c3}, we take the function $\varphi$ in Theorem \ref{t4} to be
\be{
\varphi(x)=d_{T_v-v}^2(x),
}
where $T_v-v=\{u-v: u\in T_v\}$.
We will prove the following lemma.
\begin{lemma}\label{l4}
For this $\varphi$, we have that $\varphi$ is differentiable, $D\varphi: \mathbb{R}^d\to \mathbb{R}^d$ is a Lipschitz function,
\ben{\label{41}
v\cdot D\varphi(x)=\sup_{\mu\in \mathcal{P}} \{\mu \cdot D\varphi(x)\},\ \text{for all}\ x\in \mathbb{R}^d,
}
and
\ben{\label{42}
\lambda_*=2.
}
\end{lemma}
On the basis of this lemma, we can take $\mu_i=v$ for all $i$ in Theorem \ref{t4}.
This implies the following:
\be{
\Big|\E\big\{\varphi\big[\frac{\sum_{i=1}^n (X_i-v)}{n}\big]\big\}\Big|\leq \frac{ \sup_{\theta\in \Theta} E_\theta\big[|X_1-\mu_\theta|^2\big]+diam^2(\mathcal{P})}{n}.
}
The left-hand side is precisely $d_{T_v}^2(\frac{\sum_{i=1}^n X_i}{n})$; hence, we obtain \eq{c3}.
\end{proof}
We now prove Lemma \ref{l4}.
\begin{proof}[Proof of Lemma \ref{l4}]
Without loss of generality, we assume that $v=0$; hence, $T_v-v=T_v=T_0$.
For each $x$ such that $d(x,T_0)>0$, define
\be{
x_0=\text{arginf}_{y\in T_0}|x-y|.
}
Because of the convexity of $T_0$,
$x_0$ is unique for each $x$, and moreover, $x_0$ as a function of $x$ is continuous.
Based on this definition,
\be{
\varphi(x)=|x-x_0|^2.
}
Let $\mathcal{E}$ and $\mathcal{S}$ denote the set of ``edges" and ``surfaces" of $T_0$, respectively.
The $d$-dimensional set $\mathcal{R}=\{x: d(x, T_0)>0\}$ can be divided into a finite number of disjoint parts as
\be{
\mathcal{R}=\mathcal{R}_0 \cup (\cup_{e\in \mathcal{E}} \mathcal{R}_e)\cup (\cup_{s\in \mathcal{S}} \mathcal{R}_s),
}
where
\be{
\mathcal{R}_0=\{x\in \mathcal{R}: x_0=0\},
}
\be{
\mathcal{R}_e=\{x\in \mathcal{R}: x_0\in e, x_0\ne 0\},
}
and
\be{
\mathcal{R}_s=\{x\in \mathcal{R}: x_0\in s, x_0\notin e \ \text{for any}\ e\in \mathcal{E}\}.
}
For each $x\in \mathcal{R}_s$, we change the coordinates such that $x_0$ is the origin and regard $\mathbb{R}^d$ as $s^\perp \bigotimes s$, where $s^\perp$ is the orthogonal space of $s$.
Suppose that $s^\perp$ is $d_1$-dimensional.
Then, under this new coordinate system and for $y\in \mathcal{R}_s$, we have
\be{
\varphi(y)=y_1^2+\dots + y_{d_1}^2.
}
Hence
\be{
D \varphi(y)=2(y_1,\dots, y_{d_1}, 0,\dots, 0)^T=2(y-y_0),
}
$D^2 \varphi(y)$ is a diagonal matrix with the first $d_1$ diagonal entries being $2$ and the rest being $0$, and $||D^2\varphi(y)||_{op}\leq 2.$
Similar arguments and results apply to $x\in \mathcal{R}_e$ and to $x\in \mathcal{R}_0$.
Recall that $y_0$ is a continuous function of $y$. We conclude that $D\varphi$ is continuous.
Therefore, we have \eq{42}.

We now prove \eq{41}.
Recall that we assumed that $v=0$.
On one hand, as $0\in \mathcal{P}$,
$$\sup_{\mu\in \mathcal{P}}\mu \cdot D\varphi(x)\geq 0.$$
On the other hand, by considering $x\in \mathcal{R}_0, \mathcal{R}_e, \mathcal{R}_s$ separately as above,
as $\mu\in \mathcal{P}$ points ``inwards" and $D\varphi(x)=2(x-x_0)$ points ``outwards", it is clear that
\be{
\mu \cdot D\varphi(x)\leq 0,
}
which proves \eq{41}.
\end{proof}

\begin{proof}[Proof of Remark \ref{r3}] 
Let $\mathbf{B}_0(R)$ denote a disk of radius $R$ in a plane.
For $m\in \mathbb{N}$, denote as $P_m$ a  regular $m$-sided polygon with $\mathbf{B}_0(R)$ as the inscribed circle (see Figure 1 below).
\begin {center}
\begin{tikzpicture}[scale=1.5]
\draw (0,0) circle (0.866)[thick];

\foreach \i in {0,1,2,3,4,5}
{
\draw (\i * 60:1) edge[thick]
({(\i + 1) * 60}:1);
}

\draw (120:1) edge [left] node {\tiny $r_m$}
(0,0);

\draw (0,0.866) edge [right] node {\tiny R}
(0,0);

\node at (0, -1.3) {Figure 1: $\mathcal{P}$ \& $P_m$};
\end{tikzpicture}

\end {center}
Write $r_m$ as the radius of the regular $m$-sided polygon. Then, $r_m=\frac{R}{\cos\frac{\pi}{m}}$. 
We can easily check that
\[r_m-R\leq \frac{7\pi^2R}{m^2} \ \emph{for} \  m\ge 3\] and
\[\lim_{m\rightarrow+\infty}m^2(r_m-R)= \frac{\pi^2R}{2}.\] 
Now, we expand the set $\Theta$ as $\Theta_m$ such that  $\{E_\theta [X_1]: \theta\in \Theta_m\}=P_m$ and
\[\sup_{\theta\in \Theta_m} E_\theta[|X_1-E_\theta[X_1]|^2]=\sup_{\theta\in \Theta} E_\theta[|X_1-E_\theta[X_1]|^2]=:\bar{\sigma}^2.\]
Set  $\mathbb{E}_m=\sup_{\theta\in \Theta_m}E_\theta$. Then,
\begin {eqnarray*}
\mathbb{E}[d_{\mathcal{P}}(\overline{X}_n)]&\le& \mathbb{E}_m[d_{\mathcal{P}}(\overline{X}_n)]\\
&\le& \mathbb{E}_m[d_{P_m}(\overline{X}_n)]+r_m-R\\
&\le& r_m-R+\sqrt{\frac{m}{n}}\sqrt{\overline{\sigma}^2+4r^2_m}.
\end {eqnarray*}
By setting $m=n^{\frac{1}{5}}$, we have
\[\mathbb{E}[d_{\mathcal{P}}(\overline{X}_n)]\le (7\pi^2R+\sqrt{\overline{\sigma}^2+16R^2}) n^{-\frac{2}{5}}\]
and
\[\limsup_{n\rightarrow+\infty}\Big(n^{\frac{2}{5}}\mathbb{E}[d_{\mathcal{P}}(\overline{X}_n)]\Big)\le \frac{\pi^2R}{2}+\sqrt{\overline{\sigma}^2+4R^2}.\]

\end {proof}

\subsection{Proofs in Section 3}

In this subsection, we first introduce Stein's method, which is our main tool for proving the results presented in Section 3. Then, we prove Theorem \ref{t1}. Finally, we discuss the modification of the proof of Theorem \ref{t1} for obtaining Theorem \ref{t6}.

\subsubsection{Stein's method for distributional approximations}

Stein's method was introduced by \cite{St72} for distributional approximations.
The book by \cite{ChGoSh10} contains an introduction to Stein's method and many recent advances.
Here, we will explain the basic ideas in the context of normal approximation.

Let $W$ be a random variable with mean $x$ and variance $t>0$, and let $Z_{x, t}\sim N(x, t)$ be a Gaussian random variable.
The Wasserstein distance between their distributions is defined as
\ben{\label{22}
\sup_{\varphi\in lip(\mathbb{R}): ||\varphi'||\leq 1} \big\{E [\varphi(W)]-E[\varphi(Z_{x,t})]\big\}.
}
Inspired by the fact that $Y\sim N(x, t)$ if and only if
\ben{\label{21}
E[(Y-x)f(Y)]=t E[f'(Y)]
}
for all absolutely continuous functions $f$ for which the above expectations exist,
we consider the following Stein equation:
\ben{\label{1}
t f'(w)-(w-x)f(w)=\varphi(w)-E\varphi(Z_{x, t}).
}
A bounded solution to \eq{1} is known to be
\ben{\label{2}
f_{\varphi}(w)=\frac{1}{\sqrt{t}}e^{\frac{(w-x)^2}{2t}}\int_{-\infty}^{\frac{w-x}{\sqrt{t}}} e^{-y^2/2}\Big\{\varphi(x+\sqrt{t} y)-E\big[\varphi(x+\sqrt{t} Z)\big]\Big\} dy.
}
Hereafter, we denote the standard Gaussian random variable $Z_{0,1}$ as $Z$.
Setting $w=W$ and taking the expectation on both sides of \eq{1}, we have
\besn{\label{23}
&\sup_{\varphi\in lip(\mathbb{R}): ||\varphi'||\leq 1} \big\{ E[\varphi(W)]-E[\varphi(Z_{x,t})] \big\}\\
=&\sup_{\varphi\in lip(\mathbb{R}): ||\varphi'||\leq 1} E\big\{ t f_{\varphi}'(W)-(W-x)f_{\varphi}(W)\big\}.
}
The Wasserstein distance between the distribution of $W$ and $N(x,t)$ is then bounded  by using the properties of $f_\varphi$ and by exploiting the dependence structure of $W$.

We will need to use the following properties of $f_\varphi$.
The first lemma provides an upper bound for $f''_\varphi$.
\begin{lemma}\label{l2}
For the solution \eq{2} to Stein's equation \eq{1}, we have
\ben{\label{25}
||f_\varphi''|| \leq \frac{2}{t}||\varphi'||.
}
\end{lemma}

\begin{proof}
Define
\be{
g(s):=\sqrt{t} f_\varphi(x+\sqrt{t} s),\quad h(y):=\varphi(x+\sqrt{t} y).
}
We have
\be{
g(s)= e^{s^2/2} \int_{-\infty}^s e^{-y^2/2} \big\{h(y)-E [h(Z)] \big\} dy.
}
It is known that
$g(s)$ is a bounded solution to
\be{
g'(s)-sg(s)=h(s)-E [h(Z)]
}
and [see,  for example, (2.13) of \cite{ChGoSh10}]
\be{
||g''||\leq 2 ||h'||.
}
This implies \eq{25}.
\end{proof}

It is known that $V(t,x):=E\varphi(x+\sqrt{t}Z)$ is the solution to the heat equation
\ben{\label{24}
\partial_t V (t, x)= \frac{1}{2} \partial^2_{xx} V(t, x), \ V(0,x)=\varphi(x).
}
The next lemma relates the solution to the Stein equation to the solution to the heat equation.
\begin{lemma}\label{l3}
Let $V(\cdot, \cdot)$ be the solution to the heat equation \eq{24}.
Let $f_\varphi$ be the solution \eq{2} to Stein's equation \eq{1}.
We have
\ben{\label{26}
E [f_\varphi(x+\sqrt{t}Z)] = - \partial_x V(t,x).
}
\end{lemma}

\begin{proof}
Define again
\be{
g(s):=\sqrt{t} f_\varphi(x+\sqrt{t} s),\quad h(y):=\varphi(x+\sqrt{t} y).
}
We have
\be{
g(s)= e^{s^2/2} \int_{-\infty}^s e^{-y^2/2} \big\{h(y)-E [h(Z)] \big\} dy.
}
and $g(s)$ is a bounded solution to
\ben{\label{27}
g'(s)-sg(s)=h(s)-E [h(Z)].
}
As
\be{
E [f_\varphi(x+\sqrt{t}Z)]=\frac{1}{\sqrt{t}} E [g(Z)] \ \text{and}\  - \partial_x V(t,x)=-\frac{1}{\sqrt{t}} E [h'(Z)],
}
to prove \eq{26}, we only need to show
\ben{\label{28}
E[g(Z)]=-E[h'(Z)].
}
From (2.87) of \cite{ChGoSh10}, we have
\bes{
g(s)=&-\sqrt{2\pi} e^{s^2/2} (1-\Phi(s))\int_{-\infty}^s h'(u) \Phi(u)du\\
&-\sqrt{2\pi} e^{s^2/2} \Phi(s) \int_s^{\infty} h'(u)[1-\Phi(u)] du,
}
where $\Phi(\cdot)$ denotes the standard normal distribution function.
We have
\bes{
E[g(Z)]=& \int_{-\infty}^{\infty} (1-\Phi(s)) \int_{-\infty}^s (-h'(u)) \Phi(u) du ds\\
&+\int_{-\infty}^{\infty} \Phi(s) \int_{s}^\infty (-h'(u)) (1-\Phi(u)) du ds\\
=& \int_{-\infty}^\infty (-h'(u)) \Big\{\Phi(u) \int_u^\infty (1-\Phi(s))ds+(1-\Phi(u)) \int_{-\infty}^u \Phi(s)ds\Big\} du.
}
Let $\phi(u)$ be the standard normal density function. We have
\bes{
& \Phi(u)\int_u^\infty \int_s^{\infty} \phi(v) dv ds+(1-\Phi(u))\int_{-\infty}^ u \int_{\infty}^s \phi(v) dvds\\
=& \Phi(u)\int_u^\infty (v-u) \phi(v) dv +(1-\Phi(u)) \int_{-\infty}^u (u-v) \phi(v) dv\\
=&\Phi(u)\big[\phi(u)-u(1-\Phi(u))\big]+(1-\Phi(u)) \big[u\Phi(u)+\phi(u)\big]\\
=&\phi(u).
}
Therefore,
\be{
E[g(Z)]=\int_{-\infty}^\infty (-h'(u)) \phi(u)du=-E [h'(Z)].
}
This proves \eq{28} and hence, the lemma.
\end{proof}

\subsubsection{Proofs of Theorems \ref{t1} and \ref{t6}}

\begin{proof}[Proof of Theorem \ref{t1}]
The proof is by Lindeberg's swapping argument and Stein's method.
The approach was used by \cite{Ro17} for a martingale CLT.
See also \cite{So17}.

We note that in general $X_i$ is not independent of $\{X_j: j\ne i\}$. This fact prevents us from using some of the techniques in Stein's method.

Without loss of generality, we assume that $||\varphi'||=1.$
Denote
\be{
W_0=0, \ W_k=\xi_1+\dots +\xi_{k}, \ \xi_i=\frac{X_i-\mu_i}{\sigma_i\sqrt{n}},
}
and denote
\be{
W_{[k]}:=\{W_1,\dots, W_k\}.
}
For arbitrary random vectors $X$ and $Y$, denote
\be{
\E^X[\varphi(X, Y)]:=\big\{\E[\varphi(x,Y)] \big\}_{x=X}.
}
We will prove the following claim.

\begin{claim}\label{claim2}
Let $\phi_{\sigma}(\cdot)$ be the density function of $N(0,\sigma^2)$ and let $*$ denote the convolution of functions. 
For any k=1,\dots, n, we have
\be{
\Big| \E^{W_{[k-1]}} \Big[\varphi*\phi_{\sqrt{\frac{n-k}{n}}}(W_k)-\varphi*\phi_{\sqrt{\frac{n-k+1}{n}}}(W_{k-1})\Big]    \Big| \leq \frac{C_1}{(n-k+1)\sqrt{n}},
}
where $C_1$ is as in the statement of Theorem \ref{t1}.
\end{claim}
Using telescoping sum and the independence assumption and applying Claim \ref{claim2} recursively from $k=n$ to $k=1$ as in the argument below Claim \ref{claim1}, we have
\bes{
&\Big| \E[\varphi(W_n)]-E[\varphi(Z)] \Big|\\
=&\Big| \E\Big\{ \sum_{k=1}^n [\varphi*\phi_{\sqrt{\frac{n-k}{n}}}(W_k)-\varphi*\phi_{\sqrt{\frac{n-k+1}{n}}}(W_{k-1})]   \Big\}\Big|\\
\leq & \sum_{k=1}^n  \frac{C_1}{(n-k+1)\sqrt{n}}\\
\leq & \frac{C_1(\log n+1)}{\sqrt{n}}.
}
Therefore, we obtain Theorem \ref{t1}, subject Claim \ref{claim2}.

To prove Claim \ref{claim2},
let $\eta_1,\dots, \eta_n$ be an i.i.d.\  sequence of random variables distributed as $N(0, \frac{1}{n})$ and be independent of $\{X_1, \dots, X_n\}$,
and let
\be{
T_k=\eta_n+\dots +\eta_{n-k+1}\sim N(0,\frac{k}{n}).
}
We have
\besn{\label{4}
&\E^{W_{[k-1]}} \Big[\varphi*\phi_{\sqrt{\frac{n-k}{n}}}(W_k)-\varphi*\phi_{\sqrt{\frac{n-k+1}{n}}}(W_{k-1})\Big]\\
=& \E^{W_{k-1}} \big\{ \varphi(W_k+T_{n-k}) -E[\varphi(Z_{W_{k-1}, \frac{n-k+1}{n}})] \big\},
}
where as in Section 5.2.1, $Z_{x,t}\sim N(x,t)$.
Given $W_{k-1}$, let $f$ be the solution to (cf. \eq{1})
\ben{\label{6}
\frac{n-k+1}{n} f'(w)-(w-W_{k-1} )f(w)=\varphi(w)-E\big[\varphi(Z_{W_{k-1}, \frac{n-k+1}{n}})\big].
}
Based on Lemma \ref{l2} and $||\varphi'||= 1$,
\ben{\label{7}
||f''||\leq \frac{2n}{n-k+1}.
}
From \eq{6}, we can rewrite \eq{4} as
\besn{\label{32}
&\E^{W_{[k-1]}} \Big[\varphi*\phi_{\sqrt{\frac{n-k}{n}}}(W_k)-\varphi*\phi_{\sqrt{\frac{n-k+1}{n}}}(W_{k-1})\Big]\\
=&\E^{W_{k-1}}\Big[\frac{n-k+1}{n} f'(W_k+T_{n-k} )-(\xi_{k} +T_{n-k} )f(W_k+T_{n-k})\Big]\\
=&\E^{W_{k-1}}\Big[\frac{1}{n} f'(W_k+T_{n-k} )-\xi_{k}  f(W_k+T_{n-k})\\
&\qquad \quad +\frac{n-k}{n} f'(W_k+T_{n-k} )-T_{n-k} f(W_k+T_{n-k})\Big].
}
Recall that $T_{n-k}\sim N(0, \frac{n-k}{n})$ and is independent of $\{X_1,\dots, X_n\}$.
Using \eq{21} with $Y=T_{n-k}$, $x=0, t=(n-k)/n$, we have
\be{
\E^{W_{k-1}}\Big[\frac{n-k}{n} f'(W_k+T_{n-k} )-T_{n-k} f(W_k+T_{n-k})\Big]=0,
}
and
\be{
-\E^{W_{k-1}}\Big\{-\big[\frac{n-k}{n} f'(W_k+T_{n-k} )-T_{n-k} f(W_k+T_{n-k})\big]\Big\}=0.
}
Therefore, from \eq{32} and \eq{33} where we regard $Y$ as the third and fourth terms on the right-hand side of \eq{32}, we obtain
\besn{\label{104}
&\E^{W_{[k-1]}} \Big[\varphi*\phi_{\sqrt{\frac{n-k}{n}}}(W_k)-\varphi*\phi_{\sqrt{\frac{n-k+1}{n}}}(W_{k-1})\Big]\\
=&\E^{W_{k-1}}\Big[\frac{1}{n} f'(W_k+T_{n-k} )-\xi_{k}  f(W_k+T_{n-k})\Big].
}
Rewrite
\besn{\label{34}
&\E^{W_{k-1}}\Big[\frac{1}{n} f'(W_k+T_{n-k} )-\xi_{k}  f(W_k+T_{n-k})\Big]\\
=&\E^{W_{k-1}}\Big\{\frac{1}{n} f'(W_k+T_{n-k} )-\xi_{k}  f(W_{k-1}+T_{n-k+1})\\
&\qquad \quad -\xi_{k}[f(W_{k-1}+\xi_{k}+T_{n-k})-  f(W_{k-1}+T_{n-k+1})]\Big\}\\
=&\E^{W_{k-1}}\Big[\frac{1}{n} f'(W_{k-1}+T_{n-k+1} )+R_1-\xi_{k}  f(W_{k-1}+T_{n-k+1})\\
&\qquad \quad -\xi_{k}^2 f'(W_{k-1}+T_{n-k+1}) + \xi_{k} \eta_{k} f'(W_{k-1}+T_{n-k+1})+R_2   \Big]\\
=&\E^{W_{k-1}}\Big[\frac{1}{n} f'(W_{k-1}+T_{n-k+1} )+R_1-\xi_{k}  f(W_{k-1}+T_{n-k+1})\\
&\qquad \quad -\xi_{k}^2 f'(W_{k-1}+T_{n-k+1}) + \xi_{k} \eta_{k} f'(W_{k-1}+T_{n-k})+R_2+R_3   \Big],
}
where
\be{
R_1=\frac{1}{n} f'(W_k+T_{n-k} )-\frac{1}{n} f'(W_{k-1}+T_{n-k+1} ),
} 
\be{
R_2=-\xi_k\Big[f(W_{k-1}+\xi_{k}+T_{n-k})-  f(W_{k-1}+T_{n-k+1})-(\xi_k-\eta_k)f'(W_{k-1}+T_{n-k+1})\Big]
} 
and
\be{
R_3=\xi_k \eta_k\Big[f'(W_{k-1}+T_{n-k+1})- f'(W_{k-1}+T_{n-k})\Big].
} 
Based on \eq{7} and the fact that $X_k$ is independent of $W_{k-1}$ and $\eta_k\sim N(0,\frac{1}{n})$ is independent of $\{X_1,\dots, X_n\}$, we have
\be{
\E^{W_{k-1}}[|R_1|]\leq \E^{W_{k-1}} \Big[\frac{1}{n}|\xi_{k}-\eta_{k}| ||f''||\Big]\leq \frac{2}{n-k+1} \Big(\frac{1}{\sqrt{n}}+\E[|\xi_k|]\Big),
}
\be{
\E^{W_{k-1}}[|R_2|]\leq \E^{W_{k-1}} \Big[|\xi_{k}| \frac{(\xi_{k}-\eta_{k})^2}{2} ||f''||\Big]\leq
\frac{2n}{n-k+1} \Big(\E[|\xi_k|^3]/2+\frac{1}{n}\E[|\xi_k|]/2\Big),
}
\be{
\E^{W_{k-1}}[|R_3|]\leq \E^{W_{k-1}} \Big[|\xi_{k}| \eta_{k}^2\Big] ||f''||\leq \frac{2}{n-k+1} \E[|\xi_k|].
}
From \eq{104}, \eq{34} and the estimates above, we have
\besn{\label{105}
&\Big| \E^{W_{[k-1]}} \Big[\varphi*\phi_{\sqrt{\frac{n-k}{n}}}(W_k)-\varphi*\phi_{\sqrt{\frac{n-k+1}{n}}}(W_{k-1})\Big]-A\Big| \\ 
\leq & \frac{2}{(n-k+1)\sqrt{n}}\Big( 1 +\frac{5 n^{1/2}}{2} \E[|\xi_k|] +\frac{n^{3/2}}{2} \E[|\xi_k|^3]     \Big),
}
where 
\besn{\label{106}
A:=&\E^{W_{k-1}} \Big\{ \xi_{k} [-f(W_{k-1}+T_{n-k+1})] + (\xi_{k}^2-\frac{1}{n}) [-f'(W_{k-1}+T_{n-k+1})]\\
&\qquad\quad  +\xi_{k} \eta_{k} f'(W_{k-1}+T_{n-k})\Big\}.
}
Note that
\be{
n^{1/2}\E[|\xi_k|] =\sup_{\mu_i, \sigma_i\atop i=1,\dots, n} \E\Big[|\frac{X_i-\mu_i}{\sigma_i}|\Big]\leq \frac{\overline{\sigma}+(\overline{\mu}-\underline{\mu})}{\sigma_0},
}
and
\be{
n^{3/2}\E[|\xi_k|^3]=\sup_{\mu_i, \sigma_i\atop i=1,\dots, n} \E\Big[|\frac{X_i-\mu_i}{\sigma_i}|^3\Big]\leq \frac{4[\overline{\gamma}+(\overline{\mu}-\underline{\mu})^3]}{\sigma_0^3}.
}
Therefore, \eq{105} is further bounded by 
\be{
\frac{C_1}{(n-k+1)\sqrt{n}},
}
where $C_1$ is as in the statement of Theorem \ref{t1}.

We are left to show that $A$ in \eq{106} equals 0.
Since $\eta_k$ has mean 0 and is independent of $\{X_1,\dots, X_n\}$ and $T_{n-k}$, we have
\be{
\E^{W_{k-1}} \Big[\xi_{k} \eta_{k} f'(W_{k-1}+T_{n-k})\Big]=0\ \text{and} \ \E^{W_{k-1}}\Big[-\xi_k \eta_k f'(W_{k-1}+T_{n-k})\Big]=0.
}
By the property \eq{33} of sublinear expectation, we have
\be{
A =\E^{W_{k-1}} \Big\{ \xi_{k} [-f(W_{k-1}+T_{n-k+1})] + (\xi_{k}^2-\frac{1}{n}) [-f'(W_{k-1}+T_{n-k+1})]\Big\}.
}
Using Lemma \ref{l3} and $t_i=\frac{n-i}{n}$ in the statement of the theorem, we have
\be{
A=\E^{W_{k-1}} \Big[ \xi_{k} \partial_x V(t_{k-1}, W_{k-1}) + (\xi_{k}^2-\frac{1}{n}) \partial^2_{xx} V(t_{k-1}, W_{k-1})\Big].
}
Moreover, by the definition of $\xi_k$ and $V_{i}$ below \eq{17}, we have
\be{
A=\E^{W_{k-1}} \Big\{   \frac{X_k-\mu_k}{\sigma_k} \frac{\partial_x V_{k-1}}{\sqrt{n}} +\big[\frac{(X_k-\mu_k)^2}{\sigma_k^2}-1\big] \frac{\partial_{xx}^2 V_{k-1}}{n}  \Big\},
}
and by the definition of $\E$,
\be{
A=\sup_{(\mu,\sigma^2)\in \mathcal{M}_2} \Big\{   \frac{\mu-\mu_k}{\sigma_k} \frac{\partial_x V_{k-1}}{\sqrt{n}} +\big[\frac{\sigma^2+(\mu-\mu_k)^2}{\sigma_k^2}-1\big] \frac{\partial_{xx}^2 V_{k-1}}{n}   \Big\}.
}
Finally, by the choice of $\mu_k$ and $\sigma_k$ in \eq{13} and \eq{14}, we have $A=0$.
Note that part of the reason for the particular expansion of \eq{34} is to find connections to $V$.
This, together with \eq{105}, proves Claim \ref{claim2}.

\end{proof}

\begin{proof}[Proof of Theorem \ref{t6}]
The proof is similar to that of Theorem \ref{t1}.
We use a slightly different expansion (cf. \eq{111}) and make use of the convexity (concavity) of $\varphi$ (cf. \eq{113} and \eq{114}).

We only prove the case where $\varphi$ is convex.
The concave case follows from a similar argument.
Without loss of generality, we assume that $\mu=0$ and $||\varphi'|| = 1$.
Denote
\be{
W_0=0,\ W_k=\xi_1+\dots +\xi_{k},\quad \xi_i=\frac{X_i}{\overline{B}_n},
}
and denote
\be{
W_{[k]}:=\{W_1,\dots, W_k\}.
}
Define
\be{
\Sigma_k^2=\sum_{i=n-k+1}^n \overline{\sigma}_i^2/\overline{B}_n^2
}
We will prove the following claim.
\begin{claim}\label{claim3}
Let $\phi_{\sigma}(\cdot)$ be the density function of $N(0,\sigma^2)$ and let $*$ denote the convolution of functions. 
For any k=1,\dots, n, we have
\be{
\Big| \E^{W_{[k-1]}} \Big[\varphi*\phi_{\Sigma_{n-k}}(W_k)-\varphi*\phi_{\Sigma_{n-k+1}}(W_{k-1})\Big]    \Big| \leq \frac{1}{\Sigma_{n-k+1}^2}\Big(\frac{2\overline{\sigma}_k^3}{\overline{B}_n^3}+ \E[|\xi_k|^3]\Big).
}
\end{claim}
Using telescoping sum and the independence assumption and applying Claim \ref{claim3} recursively from $k=n$ to $k=1$ as in the argument below Claim \ref{claim2}, we obtain the theorem.

To prove Claim \ref{claim3}, let $\eta_1,\dots, \eta_n$ be an independent sequence of random variables distributed as $\eta_i\sim N(0, \frac{\overline{\sigma}_i^2}{\overline{B}_n^2})$ and be independent of $\{X_1, \dots, X_n\}$,
and let
\be{
T_k=\eta_n+\dots +\eta_{n-k+1}\sim N(0,\Sigma_k^2).
}
As in \eq{4}, we have
\bes{
&\E^{W_{[k-1]}} \Big[\varphi*\phi_{\Sigma_{n-k}}(W_k)-\varphi*\phi_{\Sigma_{n-k+1}}(W_{k-1})\Big]\\
=& \E^{W_{k-1}} \big\{ \varphi(W_k+T_{n-k}) -E[\varphi(Z_{W_{k-1}, \Sigma_{n-k+1}^2})] \big\}.
}
Given $W_{k-1}$, let $f$ be the solution to
\ben{\label{110}
\Sigma_{n-k+1}^2 f'(w)-(w-W_{k-1} )f(w)=\varphi(w)-E\Big[\varphi(Z_{W_{k-1}, \Sigma_{n-k+1}^2})\Big].
}
Based on lemma \ref{l2} and $||\varphi'||= 1$, we have
\ben{\label{7p}
||f''||\leq \frac{2}{\Sigma_{n-k+1}^2}.
}
By a similar argument leading to \eq{104}, we have
\bes{
&\E^{W_{[k-1]}} \Big[\varphi*\phi_{\Sigma_{n-k}}(W_k)-\varphi*\phi_{\Sigma_{n-k+1}}(W_{k-1})\Big]\\
=&\E^{W_{k-1}}\Big[\frac{\overline{\sigma}_k^2}{\overline{B}_n^2} f'(W_k+T_{n-k} )-\xi_{k}  f(W_k+T_{n-k})\Big].
}

The appropriate change to \eq{34} is as follows:
\besn{\label{111}
&\E^{W_{k-1}}\Big[\frac{\overline{\sigma}_k^2}{\overline{B}_n^2} f'(W_k+T_{n-k} )-\xi_{k}  f(W_k+T_{n-k})\Big]\\
=&\E^{W_{k-1}}\Big\{\frac{\overline{\sigma}_k^2}{\overline{B}_n^2} f'(W_{k-1}+T_{n-k} )+ R_1 \\
&-\xi_{k}  f(W_{k-1}+T_{n-k})+\xi_k^2 (-f'(W_{k-1}+T_{n-k}))+R_2\Big\},
}
where 
\be{
R_1=\frac{\overline{\sigma}_k^2}{\overline{B}_n^2} f'(W_k+T_{n-k} )-\frac{\overline{\sigma}_k^2}{\overline{B}_n^2} f'(W_{k-1}+T_{n-k} )
}
and
\be{
R_2=-\xi_k \Big[ f(W_k+T_{n-k})-f(W_{k-1}+T_{n-k})-\xi_k f'(W_{k-1}+T_{n-k})\Big].
}
Based on \eq{7p} and the fact that $X_k$ is independent of $W_{k-1}$ and $\eta_k$ is independent of $\{X_1,\dots, X_n\}$,
\be{
\E^{W_{k-1}}[|R_1|]\leq \E^{W_{k-1}} \Big[\frac{\overline{\sigma}_k^2}{\overline{B}_n^2}|\xi_{k}| ||f''||\Big]\leq \frac{2}{\Sigma_{n-k+1}^2} \cdot \frac{\overline{\sigma}_k^2}{\overline{B}_n^2}\E[|\xi_k|]
\leq \frac{2}{\Sigma_{n-k+1}^2}\cdot \frac{\overline{\sigma}_k^3}{\overline{B}_n^3},
}
\be{
\E^{W_{k-1}}[|R_2|]\leq \E^{W_{k-1}} \Big[\frac{|\xi_{k}|^3}{2}\Big] ||f''||\leq
\frac{1}{\Sigma_{n-k+1}^2} \E[|\xi_k|^3].
}
Therefore, we have
\besn{\label{108}
&\bigg| \E^{W_{k-1}}\Big[\frac{\overline{\sigma}_k^2}{\overline{B}_n^2} f'(W_k+T_{n-k} )-\xi_{k}  f(W_k+T_{n-k})\Big] -B\bigg|\\
\leq & \frac{1}{\Sigma_{n-k+1}^2}\Big(\frac{2\overline{\sigma}_k^3}{\overline{B}_n^3}+ \E[|\xi_k|^3]\Big),
}
where
\be{
B:=\E^{W_{k-1}}\Big\{\xi_{k} [ -f(W_{k-1}+T_{n-k})]+(\xi_k^2-\frac{\overline{\sigma}_k^2}{\overline{B}_n^2}) [-f'(W_{k-1}+T_{n-k})]\Big\}.
}
By the definition of $\xi_k$, we have
\be{
B=\E^{W_{k-1}}\Big\{\frac{X_k}{\overline{B}_n} [ -f(W_{k-1}+T_{n-k})]+\frac{X_k^2-\overline{\sigma}_k^2}{\overline{B}_n^2} [-f'(W_{k-1}+T_{n-k})]\Big\}.
}
Since we have assumed that $\E(X_k)=\E(-X_k)=0$, we have, using the property \eq{33} of the sublinear expectation and also the fact that $T_{n-k}$ is independent of $\{X_1,\dots, X_n\}$,
\be{
B=\E^{W_{k-1}}\Big\{\frac{X_k^2-\overline{\sigma}_k^2}{\overline{B}_n^2} [-f'(W_{k-1}+T_{n-k})]\Big\}.
}
From Lemma \ref{l3} and the fact that $T_{n-k}$ is independent of $\{X_1,\dots, X_n\}$, we have
\ben{\label{113}
B=\E^{W_{k-1}}\Big\{\frac{X_k^2-\overline{\sigma}_k^2}{\overline{B}_n^2} \partial^2_{xx}V(\Sigma_{n-k}^2, W_{k-1}) \Big\}.
}
Since we have assumed that $\varphi$ is convex, the solution to the PDE \eq{102} (cf. \eq{107}) is also convex in the argument $x$, that is, $\partial^2_{xx} V\geq 0$. Therefore, by the definition of sublinear expectation, 
\ben{\label{114}
B=0,
}
and hence by \eq{108},
\bes{
&\bigg| \E^{W_{k-1}}\Big[\frac{\overline{\sigma}_k^2}{\overline{B}_n^2} f'(W_k+T_{n-k} )-\xi_{k}  f(W_k+T_{n-k})\Big] \bigg|\\
\leq & \frac{1}{\Sigma_{n-k+1}^2}\Big(\frac{2\overline{\sigma}_k^3}{\overline{B}_n^3}+ \E[|\xi_k|^3]\Big),
}
This proves Claim \ref{claim3}.

\end{proof}

\subsection{Proofs in Section 4}

\begin {proof} [Proof of Theorem \ref{CLT}] Define $\xi=(\xi_1,\cdots, \xi_n):\mathbb{R}^n\rightarrow\mathbb{R}^n$ by $\xi_i(x)=x_i$, $i=1,\cdots, n$. Denote as $\mathcal{H}$ the collection of continuous real-valued functions $h$ on $\mathbb{R}^n$ with $|h(x)|\le C(1+|x|^3)$ for some constant $C>0$. For a function $h\in\mathcal{H}$, set
\[\mathbb{E}[h(\xi)]:=\sup\limits_{\sigma\in\Sigma^{\mathbb{N}}_G}E[h(X^{\sigma}_{1,n},\cdots, X^{\sigma}_{n,n})].\]
Then, $\mathbb{E}[\xi_i]=\mathbb{E}[-\xi_i]=0,$  $ \mathbb{E}[\xi_i^2]=\overline{\sigma}^2$ and $-\mathbb{E}[-\xi_i^2]=\underline{\sigma}^2$, $i=1,2,\cdots,n.$  Moreover, for a function $\varphi\in lip(\mathbb{R})$, we have  \[\mathbb{E}[\varphi(\xi_i)]=\sup\limits_{\lambda\in[\underline{\sigma},\overline{\sigma}]}E[\varphi(\lambda X_i)]=:\mathcal{N}[\varphi],\]
i.e., $\xi_1, \cdots, \xi_n$ are identically distributed under $\mathbb{E}$.

Set $W_{i,n}=\frac{\xi_1+\cdots+\xi_i}{\sqrt{n}}$. We next prove that, for any function $\varphi\in lip(\mathbb{R})$,
\begin{eqnarray}\label {se2}
\mathbb{E}[\varphi(W_{i+1,n})]=\mathbb{E}[\mathbb{E}[\varphi(s+\frac{\xi_{i+1}}{\sqrt{n}})]\big|_{s=W_{i,n}}].
\end {eqnarray}
On the one hand, we have, for any $\sigma\in \Sigma^{\mathbb{N}}_G$,
\[E[\varphi(W^{\sigma}_{i+1,n})]=E[E[\varphi(s+\sigma_{i+1}(s)\frac{X_{i+1}}{\sqrt{n}})]\big|_{s=W^{\sigma}_{i,n}}]\le E[\mathbb{E}[\varphi(s+\frac{\xi_{i+1}}{\sqrt{n}})]\big|_{s=W^{\sigma}_{i,n}}].\]
Therefore we obtain \[\mathbb{E}[\varphi(W_{i+1,n})]\le\mathbb{E}[\mathbb{E}[\varphi(s+\frac{\xi_{i+1}}{\sqrt{n}})]\big|_{s=W_{i,n}}].\]
On the other hand, for each $s\in\mathbb{R}$, we choose $\lambda^{\varphi,n}(s)\in[\underline{\sigma},\overline{\sigma}]$ such that
\[E[\varphi(s+\lambda^{\varphi,n}(s)\frac{X_{1}}{\sqrt{n}})]=\sup\limits_{\lambda\in [\underline{\sigma},\overline{\sigma}]}E[\varphi(s+\lambda\frac{X_{1}}{\sqrt{n}})]=\mathbb{E}[\varphi(s+\frac{\xi_{1}}{\sqrt{n}})].\]
Here, we are not sure about the measurability of the function $\lambda^{\varphi,n}(s)$. Therefore, we replace it by measurable approximations. Write $\Phi(s,t, X_1)=\varphi(s+\lambda^{\varphi,n}(t)\frac{X_{1}}{\sqrt{n}})$. For any two real numbers $s,t$, we have
\begin {eqnarray*}& &E[\Phi(s,s, X_1)]\\
&=&E[\Phi(t,s, X_1)]
   +\big(E[\Phi(s,s, X_1)]-E[\Phi(t,s, X_1)]\big)\\
&\le&E[\Phi(t,t, X_1)]+ L^\varphi|t-s|\\
&=&E[\Phi(s,t, X_1)]
   +\big(E[\Phi(t,t, X_1)]-E[\Phi(s,t, X_1)]\big)+L^\varphi|t-s|\\
&\le&E[\Phi(s,t, X_1)]+2L^\varphi|t-s|,
\end {eqnarray*} where $L^\varphi$ is the Lipschitz constant of the function $\varphi$.

For any $\epsilon>0$, set $\delta=\frac{\epsilon}{2L^\varphi}$ and
\[\lambda^{\varphi,n}_\epsilon(s)=\sum_{k\in\mathbb{Z}}\lambda^{\varphi,n}(k\delta)1_{(k\delta,(k+1)\delta]}(s).\]
Then, for any $s\in \mathbb{R}$,
\[E[\varphi(s+\lambda^{\varphi,n}_\epsilon(s)\frac{X_{1}}{\sqrt{n}})]\ge\mathbb{E}[\varphi(s+\frac{\xi_{1}}{\sqrt{n}})]-\epsilon.\]
For any $\sigma\in\Sigma^{\mathbb{N}}_G$ with $\sigma_{i+1}(s)=\lambda^{\varphi,n}_\epsilon(s)$, we have
\[E[\varphi(W^{\sigma}_{i+1,n})]=E[E[\varphi(s+\sigma_{i+1}(s)\frac{X_{i+1}}{\sqrt{n}})]\big|_{s=W^{\sigma}_{i,n}}]\ge E[\mathbb{E}[\varphi(s+\frac{\xi_{i+1}}{\sqrt{n}})]\big|_{s=W^{\sigma}_{i,n}}]-\epsilon.\]
Therefore,  \[\mathbb{E}[\varphi(W_{i+1,n})]\ge\mathbb{E}[\mathbb{E}[\varphi(s+\frac{\xi_{i+1}}{\sqrt{n}})]\big|_{s=W_{i,n}}].\]
Combining the above arguments, we prove equality (\ref{se2}).

Let $\tilde{\xi}_1,\cdots, \tilde{\xi}_n$ be i.i.d random variables under a sublinear expectation $\tilde{\mathbb{E}}$ with $\tilde{\xi}\sim \mathcal{N}$, the distribution of $\xi_1$. On the basis of (\ref{se2}), we have, for any $\varphi\in lip(\mathbb{R})$,
\[\mathbb{E}[\varphi(W_n)]=\tilde{\mathbb{E}}[\varphi(\frac{\tilde{\xi}_1+\cdots+\tilde{\xi}_n}{\sqrt{n}})].\] Therefore, by using Theorem 4.5 of \cite{So17}, we obtain the desired estimate.
\end {proof}

\begin {proof} [Proof of Theorem \ref{t10}] Without loss of generality, we shall only consider $\varphi$ that vanishes at infinity.
Let $u$ be the solution to the $G$-heat equation with initial value $\varphi$. Set $\sigma_\varphi(t,x)=2G(\textmd{sgn}[\partial_{xx}^2u(1-t,x)])$, $(t,x)\in [0,1)\times\mathbb{R}$, where
\begin {equation*}
\textmd{sgn}[a]=
\begin{cases}
1,  & \mbox{if }a\ge0; \\
-1, & \mbox{if }a<0.
\end{cases}
\end {equation*}
Then, $u$ satisfies
\begin {eqnarray*}
\partial_t u-\frac{1}{2}\widetilde{\sigma}_\varphi^2\partial^2_{xx} u&=&0, \ (t,x)\in (0,1]\times\mathbb{R},\\
                        u(0,x)&=& \varphi (x).
\end {eqnarray*}
By the mollification procedure, we can find $\{\sigma_n\}\subset\Sigma_G$ such that $\|\sigma_n-\sigma_\varphi\|_{L^2([0,1]\times \mathbf{B}(R))}\rightarrow0$ as $n\rightarrow\infty$ for any $R<\infty$. Next, set $v_n(t,x):=E[\varphi(W^{\sigma_n,x}_t)]$. Then, $v_n$ is the solution to the following equation:
\begin {eqnarray*}
\partial_t v_n-\frac{1}{2}\widetilde{\sigma}_n^2\partial^2_{xx} v_n&=&0, \ (t,x)\in (0,1]\times\mathbb{R},\\
                        v_n(0,x)&=& \varphi (x).
\end {eqnarray*}
As $\varphi$ vanishes at infinity, \[\mathbf{M}(R):=\mathop{\max_{|x|\ge R;}}_{1\ge t\ge 0}\big\{|u(t,x)|, |v_n(t,x)|: \ n\in\mathbb{N} \big\}\]
approaches zero as $R$ approaches $+\infty$. Also, we have  \[\mathbf{m}(\epsilon):=\max_{(t,x)\in [0,\epsilon]\times\mathbb{R}}\big\{|u(t,x)-\varphi(x)|, |v_n(t,x)-\varphi(x)|: \ n\in\mathbb{N} \big\}\]
goes to zero as $\epsilon$ goes to $0$. Set $w_n=u-v_n$ and $\varepsilon_n=\widetilde{\sigma}_n^2-\widetilde{\sigma}_\varphi^2$. Then, $w_n$, which is nonnegative, satisfies
\begin {eqnarray*}
\partial_t w_n-\frac{1}{2}\widetilde{\sigma}_n^2\partial^2_{xx} w_n&=&\frac{1}{2}\varepsilon_n\partial_{xx}^2u, \ (t,x)\in (0,1]\times\mathbb{R},\\
                        w_n(0,x)&=&0.
\end {eqnarray*}
According to the Aleksandrov-Bakel'man-Pucci-Krylov maximum principle (see, for instance, Theorem 7.1 of \cite {Lie}),
\[\sup\limits_{(t,x)\in (\epsilon,1]\times\mathbf{B}(R)}w_n\le 2\mathbf{M}(R)+2\mathbf{m}(\epsilon)+c_0 (\frac{R}{\underline{\sigma}})^{1/2}\|\varepsilon_n\partial_{xx}^2u\|_{L^2([\epsilon,1]\times \mathbf{B}(R))},\]
where $c_0$ is a universal constant. Note that, following the interior regularity of $G$-heat equation,
\[\|\varepsilon_n\partial_{xx}^2u\|_{L^2([\epsilon,1]\times \mathbf{B}(R))}\le 2\overline{\sigma}\|\partial_{xx}^2u\|_{\infty;[\epsilon,1]\times\mathbb{R}}\|\sigma_n-\sigma_\varphi\|_{L^2([0,1]\times \mathbf{B}(R))}\rightarrow0\] as $n$ approaches $+\infty$. Thus,
\[\mathbf{O}(R,\epsilon):=\limsup_{n\rightarrow\infty}\Big(\sup\limits_{(t,x)\in (\epsilon,1]\times\mathbf{B}(R)}w_n\Big)\le 2(\mathbf{M}(R)+\mathbf{m}(\epsilon))\]
and
\[\mathbf{O}(R,\epsilon)\le\lim_{R\rightarrow\infty, \epsilon\rightarrow0}\mathbf{O}(R,\epsilon)\le\lim_{R\rightarrow\infty, \epsilon\rightarrow0} 2(\mathbf{M}(R)+\mathbf{m}(\epsilon))\le0.\]
In particular, we have \[\mathcal{N}_G[\varphi]=u(0,1)=\lim_{n\rightarrow\infty}v_n(0,1)=\lim_{n\rightarrow\infty}E[\varphi(W_1^{\sigma_n})].\]
\end {proof}

\section*{Acknowledgements}

We thank the two anonymous referees for their detailed comments which led to many improvements.
Fang X. was partially supported by a CUHK start-up grant.
Peng S. was supported by NSF (No. 11626247). 
Shao Q. M. was partially supported by Hong Kong RGC GRF 14302515 and 14304917.
Song Y. was supported by NCMIS, NSFC (No. 11688101) and Key Research Program of Frontier Sciences, CAS (No. QYZDB-SSW-SYS017).

\end{document}